\documentclass{article}

\usepackage[cm]{fullpage}
\usepackage{amsmath, amssymb, amsthm, enumerate, tikz,
mathrsfs, color, comment, xcolor, algorithmicx,  mathtools} 
\usepackage{enumitem}

\usepackage[ruled]{algorithm}
\usepackage[noend]{algpseudocode}

\algrenewcommand{\algorithmiccomment}[1]{\hfill[{\textit{#1}}]}

\usepackage[colorlinks]{hyperref}
\hypersetup{linkcolor=blue, urlcolor=blue, citecolor=red}

\let\oldmarginpar\marginpar%
\renewcommand\marginpar[1]{\-\oldmarginpar[\raggedleft\footnotesize #1]%
{\raggedright\footnotesize #1}}

\newcommand{\N}{\mathbb{N}}
\newcommand{\M}{\mathscr{M}}

\newcommand{\id}{\operatorname{id}}

\renewcommand{\L}{\mathscr{L}}
\renewcommand{\H}{\mathscr{H}}

\newcommand{\J}{\mathscr{J}}

\newcommand{\R}{\mathscr{R}}

\newcommand{\set}[2]{\{#1:#2\}}
\newcommand{\n}{\{1,\ldots, n\}}
\newcommand{\genset}[1]{\langle#1\rangle}

\renewcommand{\to}{\longrightarrow}
\newcommand{\bigset}[2]{\big\{ {#1}:{#2} \big\}}

\newtheorem{thm}{Theorem}[section]
\newtheorem{lem}[thm]{Lemma}
\newtheorem{cor}[thm]{Corollary}
\newtheorem{prop}[thm]{Proposition}

\theoremstyle{definition}
\newtheorem{ex}[thm]{Example}

\newcommand{\edge}[2]{\draw (#1) -- (#2);}
\newcommand{\arc}[2]{\draw[->] (#1) -- (#2);}

\title{Computing maximal subsemigroups of a finite semigroup}
\author{C. R. Donoven, J. D. Mitchell, and W. A. Wilson}

\begin{document}

\maketitle
\begin{abstract}
  A proper subsemigroup of a semigroup is \textit{maximal} if it is not
  contained in any other proper subsemigroup.  A maximal subsemigroup of a
  finite semigroup has one of a small number of forms, as described in a paper
  of Graham, Graham, and Rhodes.  Determining which of these forms arise in a
  given finite semigroup is difficult, and no practical mechanism for doing
  so appears in the literature.  We present an algorithm for computing the
  maximal subsemigroups of a finite semigroup $S$ given knowledge of the
  Green's structure of $S$, and the ability to determine maximal subgroups of
  certain subgroups of $S$, namely its group $\mathscr{H}$-classes. 
  
  In the case of a finite semigroup $S$ represented by a generating set $X$, in
  many examples, if it is practical to compute the Green's structure of $S$
  from $X$, then it is also practical to find the maximal subsemigroups of $S$
  using the algorithm we present.  In such examples, the time taken to
  determine the Green's structure of $S$ is comparable to that taken to find
  the maximal subsemigroups.  The generating set $X$ for $S$ may consist, for
  example, of transformations, or partial permutations, of a finite set, or of
  matrices over a semiring.  Algorithms for computing the Green's structure of
  $S$ from $X$ include the Froidure-Pin Algorithm, and an algorithm of the
  second author based on the Schreier-Sims algorithm for permutation groups.
  The worst case complexity of these algorithms is polynomial in $|S|$, which
  for, say, transformation semigroups is exponential in the number of points on
  which they act. 

  Certain aspects of the problem of finding maximal subsemigroups reduce to
  other well-known computational problems, such as finding all maximal cliques
  in a graph and computing the maximal subgroups in a group.
 
  The algorithm presented comprises two parts. One part relates to computing
  the maximal subsemigroups of a special class of semigroups, known as Rees
  0-matrix semigroups. The other part involves a careful analysis of certain
  graphs associated to the semigroup $S$, which, roughly speaking, capture the
  essential information about the action of $S$ on its $\J$-classes. 
\end{abstract}

\section{Introduction}

A \textit{semigroup} $S$ is a set with an associative operation.  A
\textit{subsemigroup} $M$ of a semigroup $S$ is just a subset that is closed
under the operation of $S$. A \textit{maximal subsemigroup} $M$ of a semigroup
$S$ is a proper subsemigroup that is not contained in any other proper
subsemigroup. Every proper subsemigroup of a finite semigroup is contained in a
maximal subsemigroup, although the same is not true for infinite semigroups.
For example, the multiplicative semigroup consisting of the real numbers in the
interval $(1, \infty)$ has no maximal subsemigroups, but it does have proper
subsemigroups. 

There are numerous papers in the literature about finding maximal subsemigroups
of particular classes of semigroups; for example~\cite{Dimitrova2008aa,
Dimitrova2011aa, Dimitrova2013aa, Dimitrova2012aa, Dimitrova2009aa,
Dimitrova2009ab,  East2015aa, Gyudzhenov2006aa, Gyudzhenov2006ab, Hotzel1995aa,
Koppitz2014aa, Levi1984aa, Todorov2005aa, Yang2004aa, Yang2000aa,
Yang2001aa}. Perhaps the most important paper on this topic for finite
semigroups is that by Graham, Graham, and Rhodes~\cite{Graham1968aa}. This paper
appears to have been overlooked for many years, and indeed, special cases of the
results it contains have been repeatedly reproved.  

The main purpose of this paper is to give algorithms that can be used to find
the maximal subsemigroups of an arbitrary finite semigroup.  Our algorithms are
based on the paper from 1968 of Graham, Graham, and Rhodes~\cite{Graham1968aa}.
The algorithms described in this paper are implemented in the {\sc
GAP}~\cite{GAP4} package {\sc Semigroups}~\cite{Mitchell2017aa}.
This paper is organised as follows. In the remainder of this section, we
introduce the required background material and notation.  We state the main
results of Graham, Graham, and Rhodes~\cite{Graham1968aa} in
Propositions~\ref{prop-graham} and~\ref{prop-rms}.  In
Section~\ref{section-rees}, we describe algorithms for finding the maximal
subsemigroups of a finite regular Rees 0-matrix semigroup over a group.  In
Section~\ref{section-arbitrary}, we use the procedures from
Section~\ref{section-rees} to describe an algorithm for finding the maximal
subsemigroups of an arbitrary finite semigroup.

Henceforth we only consider finite semigroups, and we let $S$ denote an
arbitrary finite semigroup throughout.  We denote by $S ^ 1$ the semigroup
obtained from $S$ by adjoining an identity element $1\not\in S$.  In other
words, $1s = s1 = s$ for all $s \in S ^ 1$.  If $X$ is any subset of a
semigroup $S$, then we denote by $\genset{X}$ the least subsemigroup of $S$
containing $X$. If $\genset{X} = S$, then we refer to $X$ as a
\textit{generating set} for $S$.  

Let $x,y\in S$ be arbitrary.  We say that $x$ and $y$ are $\L$-related if the
principal left ideals generated by $x$ and $y$ in $S$ are equal; in other
words, $S^1x=S^1y$.  Clearly $\L$ defines an equivalence relation on $S$.  We
write $x\L y$ to denote that $x$ and $y$ are $\L$-related.  Green's
$\R$-relation is defined dually to Green's $\L$-relation; Green's $\H$-relation
is the meet, in the lattice of equivalence relations on $S$, of $\L$ and $\R$.
In any semigroup, $x\J y$ if and only if the (2-sided) principal ideals
generated by $x$ and $y$ are equal. However, in a finite semigroup  
$\J$ is the join of $\L$ and $\R$. We will refer to the equivalence
classes as $\mathscr{K}$-classes where $\mathscr{K}$ is any of $\R$, $\L$,
$\H$, or $\J$, and the $\mathscr{K}$-class of $x\in S$ will be denoted by
$K_x$. We write $\mathscr{K}^S$ if it is necessary
to explicitly refer to the semigroup on which the relation is defined. We denote
the set of $\mathscr{K}$-classes of a semigroup $S$ by $S/\mathscr{K}$.

An \emph{idempotent} is an element $x \in S$ such that $x^{2}=x$.  We denote
the set of idempotents in a semigroup $S$ by $E(S)$.  A $\J$-class of a finite
semigroup is \emph{regular} if it contains an idempotent, and a finite
semigroup is called \emph{regular} if each of its $\J$-classes is regular.
Containment of principal ideals induces a partial order on the $\J$-classes of
$S$.

If $J$ is an arbitrary
$\J$-class of a finite semigroup $S$, then we denote the \textit{principal
factor} of $J$ by $J^*$.  In other words, $J^*$ is the semigroup with elements
$J\cup \{0\}$ and multiplication $*$ given by setting $x*y=0$ if the product in
$S$ of $x,y\in J$ does not belong to $J$, and letting it have its value in $J$
otherwise. 

A \textit{Rees $0$-matrix semigroup}
$\M^{0}[I,G,\Lambda;P]$ is the set $(I\times G\times\Lambda)\cup\{0\}$ where
$I,\Lambda\neq\varnothing$, $G$ is a group, and $P$ is a $|\Lambda|\times |I|$
matrix with entries $p_{\lambda, i}$ ($\lambda \in \Lambda, i\in I$) in
$G\cup\{0\}$, with multiplication defined by

\begin{equation*}
  0x = x0 = 0 \text{\ for all\ } x \in \M^{0}[I,G,\Lambda; P]
  \qquad\text{and}\qquad 
  (i, g, j)(k, h, l) =
  \begin{cases}
    (i, g p_{j, k} h, l) & \text{if\ } p_{j, k} \neq 0,\\
    0                    & \text{if\ } p_{j, k} = 0.
  \end{cases}
\end{equation*}
Throughout this paper we assume without loss of generality that
$I \cap \Lambda= \varnothing$. 

If $J$ is a $\J$-class of a finite semigroup, then by the Rees
Theorem~\cite[Theorem 3.2.3]{Howie1995aa}, $J^{*}$ is isomorphic to either a
regular Rees $0$-matrix semigroup over a group, or a null semigroup.

A \textit{graph} is a pair $\Gamma = (V, E)$ consisting of a set of
\textit{vertices} $V$ and a set of \textit{edges} $E \subseteq
\set{\{u,v\}}{u,v\in V,\ u\not= v}$.  If $u$ is a vertex in a graph and
$\{u,v\}$ is an edge in that graph, then we say that $u$ is \textit{incident} to
$\{u,v\}$, and we say that $\{u,v\}$ is \textit{incident} to $u$. If $\Gamma =
(V, E)$ is a graph and $W$ is a subset of the vertices of $\Gamma$, then the
\textit{subgraph induced by $W$} is the graph with vertices $W$ and edges
$\set{\{u, v\}\in E}{u, v\in W}$. A \textit{clique} $K$ in a graph $\Gamma = (V,
E)$ is a subset of the vertices of $\Gamma$ such that $\{u, v\}\in E$ for all
$u, v\in K$, $u \not= v$. If $\Gamma = (V, E)$ is a graph, then the
\textit{complement} of $\Gamma$ is the graph with vertices $V$ and edges
$\set{\{u, v\}\not \in E}{u\not = v}$.  An \textit{independent set} in a graph
is a set of vertices which is a clique in the complement.  A clique is
\textit{maximal} if it is not properly contained in another clique.  A
\textit{maximal independent set} is defined analogously.  A graph is
\textit{bipartite} if its vertices can be partitioned into two independent sets.
If $\Gamma = (V, E)$ is a graph, then a \textit{path} is a non-empty sequence of
distinct vertices of $\Gamma$, $(v_{1}, \ldots, v_{m})$, where $\{v_{i},
v_{i+1}\}\in E$ for all $i \in \{1, \ldots, m - 1\}$. The \textit{connected
component} of a vertex $v$ of a graph is the set consisting of all vertices $u$
such that there is a path from $u$ to $v$. 

If $R = \M^{0}[I, G, \Lambda; P]$ is a Rees 0-matrix semigroup, then we define
the \textit{Graham-Houghton graph $\Gamma(R)$ of $R$} to be the bipartite graph
with vertices $I\cup \Lambda$ and edges $\{i, \lambda\}$ whenever $p_{\lambda,
i}\not = 0$, for $i\in I$, $\lambda\in \Lambda$.  Since we assume throughout
that $I\cap \Lambda = \varnothing$, the Graham-Houghton graph of $R$ has $|I| +
|\Lambda|$ vertices.  Variants of this graph were introduced
in~\cite{Graham1968ab, Houghton1977aa}.

We will use the following well-known results repeatedly throughout this paper.

\begin{lem}[Green's~Lemma; Lemmas~{2.2.1} and~{2.2.2} in~\cite{Howie1995aa}]
  \label{lem-green}
  Let $S$ be a semigroup, and let $x, y \in S$ be such that $x \R y$. If $s, t
  \in S^{1}$ are such that $xs = y$ and $yt = x$, then the functions $L_{x} \to
  S$ and $L_{y} \to S$ defined by $a \mapsto as$ and  $b \mapsto bt$,
  respectively, are mutually inverse bijections from the $\L$-class $L_{x}$
  onto $L_{y}$ and $L_{y}$ onto $L_{x}$, respectively, that preserve Green's
  $\R$-relation. 
\end{lem}

An analogue of Lemma~\ref{lem-green} holds when $\L$- and $\R$-relations are
interchanged, which is also referred to as Green's Lemma.

\begin{lem}[Theorem~{A.2.4} in~\cite{Rhodes2009aa}]\label{lem-d-finite}
  Let $S$ be a finite semigroup and let $x, y \in S ^ {1}$.  Then $x \J xy$ if
  and only if $x \R xy$, and  $x \J yx$ if and only if $x \L yx$. 
\end{lem}

An arbitrary (not necessarily finite) semigroup satisfying the conditions of
Lemma~\ref{lem-d-finite} is called \textit{stable}.

The following lemma is important to the description of the maximal
subsemigroups of a finite semigroup.

\begin{lem}[Proposition(1) in~\cite{Graham1968aa}]\label{lem-graham}
  Let $S$ be a finite semigroup and let $M$ be a maximal subsemigroup of $S$.
  Then $S\setminus M$ is contained in a single $\J$-class of $S$. 
\end{lem}

If $S$ is a finite semigroup and $M$ is a maximal subsemigroup of $S$, then we
denote the $\J$-class of $S$ containing $S\setminus M$ by $J(M)$.  It is also
shown in the second part of the main proposition in~\cite{Graham1968aa} that $M$
is either a union of $\H$-classes of $S$ or has non-empty intersection with
every $\H$-class of $S$. 

We include a version of the main result in~\cite{Graham1968aa}, which is
slightly reformulated for our purposes here.  More specifically,
Proposition~\ref{prop-graham}\eqref{item-prop-nonreg}
and~\eqref{item-prop-intersect} are reformulations of parts (3) and Case 1 of
part (4) of the main proposition in~\cite{Graham1968aa}.
Proposition~\ref{prop-graham}\eqref{item-prop-union} is similar to Case 2 of
part (4) of the main proposition in~\cite{Graham1968aa}.  Our statement follows
from the proof in~\cite{Graham1968aa}, but it is not exactly the statement
in~\cite{Graham1968aa}.  

\begin{prop}\label{prop-graham}
  Let $S$ be a finite semigroup, let $M$ be a maximal subsemigroup of $S$, let
  $J(M)$ denote the $\J$-class of $S$ such that $S\setminus M\subseteq J(M)$,
  and let $\phi:{J(M)}^* \to \M^{0}[I,G,\Lambda;P]$ be any isomorphism.  Then
  one of the following holds:
  \begin{enumerate}[label=\emph{(\alph*)},ref=\alph*]
    \item\label{item-prop-nonreg}
      $J(M)$ is non-regular and $J(M)\cap M = \varnothing$;

    \item\label{item-prop-intersect}
      $J(M)$ is regular, $M$ intersects every $\H$-class of $J(M)$
      non-trivially, and $(M\cap J(M))\phi \cup \{0\}
      \cong\M^{0}[I,H,\Lambda;Q]$ where
      $H$ is a maximal subgroup of $G$ and $Q$ is a $|\Lambda|\times |I|$
      matrix with entries over $H\cup\{0\}$.  In this case,
      $(M\cap J(M))\phi \cup \{0\}$ is a maximal subsemigroup of
      $\M^{0}[I,G,\Lambda;P]$;
    \item\label{item-prop-union}
      $J(M)$ is regular, $J(M)\cap M$ is a union of $\H$-classes of
      $J(M)$, and $(M\cap J(M))\phi$ equals one of the following:
      \begin{enumerate}[label=\emph{(\roman*)}, ref=\roman*, resume]
        \item\label{item-prop-rect} 
          $(I\times G\times\Lambda) \setminus (I'\times
          G\times\Lambda')$ for some $\varnothing\neq
          I'\subsetneq I$ and $\varnothing\neq \Lambda'\subsetneq \Lambda$.
          In this case, $\big((I\times G\times
          \Lambda) \setminus (I'\times G\times\Lambda')\big)\cup\{0\}$ is a
          maximal subsemigroup of $\M^0[I,G,\Lambda;P]$;

        \item\label{item-prop-cols}
          $I \times G \times \Lambda'$ for some $\varnothing \neq \Lambda'
          \subsetneq \Lambda$;

        \item\label{item-prop-rows}
          $I' \times G \times \Lambda$ for some $\varnothing \neq I' \subsetneq
          I$;

        \item\label{item-prop-empty}
          $\varnothing$.
      \end{enumerate}
  \end{enumerate}
\end{prop}

The principal difference between the main proposition of~\cite{Graham1968aa}
and Proposition~\ref{prop-graham} is that in the latter the isomorphism $\phi$
is arbitrary.  We will show how to effectively determine the maximal
subsemigroups of each type that arise in a given finite semigroup.  Certain
cases reduce to other well-known problems, such as, for example, finding all of
the maximal cliques in a graph, and computing the maximal subgroups of a finite
group. The following result describes the maximal subsemigroups of a
finite regular Rees 0-matrix semigroup over a group, and is central
to the presented algorithms.

\begin{prop}[cf. Theorem 4 in~\cite{Graham1968ab}]\label{prop-rms}
  Let $R=\M^{0}[I, G, \Lambda; P]$ be a finite regular Rees 0-matrix
  semigroup over a group $G$, and let $M$ be a subset of $R$. Then $M$ is
  a maximal subsemigroup of $R$ if and only if one of the following holds:
  \begin{enumerate}[label=\emph{(R\arabic*)}, ref=R\arabic*]
      
    \item\label{item-size-2}
      $M=\{0\}$ and $|R|=2$;\footnote{Note that in~\cite[Theorem
      4]{Graham1968ab} it is incorrectly stated that $\{0\}$ is a maximal
      subsemigroup if $R\setminus\{0\}$ is a cyclic group of prime order $p$.
      If $p > 1$, then the identity of $R\setminus\{0\}$ and $0$ comprise a
      proper subsemigroup of $R$ strictly containing $\{0\}$, which is
      therefore not maximal.}

    \item\label{item-remove-zero}
      $M=R\setminus \{0\}$ and $R\setminus\{0\}$ is a subsemigroup 
      (or equivalently there are no entries in $P$ equal to $0$);
    
    \item\label{item-remove-lambda}
      $|\Lambda| > 1$, $M = \big(I\times G\times (\Lambda\setminus
      \{\lambda\})\big) \cup\{0\}$ for some $\lambda\in \Lambda$, and there is
      at least one non-zero entry in every row and every column of ${(p_{\mu,
      i})}_{\mu\in\Lambda\setminus\{\lambda\}, i\in I}$;
    
    \item\label{item-remove-i}
      $|I| > 1$, $M=\big((I\setminus\{i\})\times G\times \Lambda\big)\cup
      \{0\}$ for some $i\in I$, and there is at least one non-zero entry in
      every row and every column of ${(p_{\lambda, j})}_{\lambda\in\Lambda,
      j\in I\setminus \{i\}}$;

    \item\label{item-rect}
      $M = \big((I\times G\times \Lambda)\setminus(I'\times
      G\times\Lambda')\big)\cup\{0\}$ where $I' = I\setminus X$, $\Lambda' =
      \Lambda\setminus Y$, and $X$ and $Y$ are proper non-empty subsets of $I$
      and $\Lambda$, respectively, such that $X\cup Y$ is a maximal independent
      set in the Graham-Houghton graph of $R$;
    
    \item\label{item-subgroup}
      $M$ is a subsemigroup isomorphic to $\M^{0}[I, H, \Lambda; Q]$ where $H$
      is a maximal subgroup of $G$ and $Q$ is a $|\Lambda|\times |I|$ matrix
      over $H\cup \{0\}$.
  \end{enumerate}
\end{prop}

The algorithms we describe require knowledge of the Green's structure of $S$. 
We briefly discuss algorithms for determining the Green's structure of a finite
semigroup $S$.  If $S$ is a regular Rees 0-matrix semigroup $\mathscr{M}^0[I,
G, \Lambda; P]$ where $G$ is a group, then it can be represented on a computer
by a generating set for $G$ and the matrix $P$. The Green's structure of such a
semigroup can be obtained directly from $P$.  For an arbitrary finite semigroup
$S$, we suppose that $S$ is represented by a generating set $X$, and that we
know how to determine the value of $xy$ for any $x,y \in S$.  The generating
set $X$ for $S$ may consist, for example, of transformations, or partial
permutations, of a finite set, or of matrices over a semiring.  Algorithms for
computing the Green's structure of $S$ from $X$ include the Froidure-Pin
Algorithm~\cite{Froidure1997aa, Jonusas2017aa}, and the algorithms described
in~\cite{East2015ab}.  The worst case complexity of these algorithms is
$O(|S||X|)$, which for transformations and partial permutations is exponential
in the number of points on which they act, and for matrices over a semiring, is
exponential in their dimension. The Froidure-Pin Algorithm determines the right
and left Cayley graphs of $S$ with respect to $X$, and  Green's $\R$- and
$\L$-relations correspond to the strongly connected components of these graphs.
There are several well-known algorithms, such as those of Tarjan or Gabow, for
finding the strongly connected components of a graph.  The algorithms
in~\cite{East2015ab} directly enumerate the $\R$- and $\L$-classes of $S$, and
so for examples of semigroups with relatively large $\J$-classes, unlike the
Froidure-Pin Algorithm, the Green's structure of $S$ can be found without
storing every element of $S$ in memory.  In this way, we will henceforth
suppose that we are able to describe the Green's structure of a finite
semigroup. If computing the Green's structure of $S$ is not practical, because,
say, it requires too much time or space, then the algorithms presented here
cannot be used to find the maximal subsemigroups of $S$. In many examples, it
appears that the converse also holds; further details can be
found in Section~\ref{section-performance}. 


\section{Regular Rees 0-matrix semigroups}\label{section-rees}

In this section, we discuss how to compute the maximal subsemigroups of a
finite regular Rees 0-matrix semigroup over a group. This has inherent
interest, but we will use the algorithms for Rees 0-matrix semigroups when
computing the maximal subsemigroups of an arbitrary finite semigroup. 

The possible types of maximal subsemigroups of a Rees 0-matrix semigroup are
described in Proposition~\ref{prop-rms}.  We begin this section by describing
how to compute the maximal subsemigroups of types~\eqref{item-size-2}
to~\eqref{item-rect}. Maximal subsemigroups of
type~\eqref{item-subgroup} are more complicated to determine, and the remainder
of this section is dedicated to describing an algorithm for finding the maximal
subsemigroups of this type. 

Throughout this section, we will denote by $R$ a finite regular Rees 0-matrix
semigroup $\M^{0}[I, G, \Lambda; P]$ over a group $G$.


\subsection{Maximal subsemigroups of types
\eqref{item-size-2}--\eqref{item-rect}}

It is trivial to check whether there exists a maximal subsemigroup of
types~\eqref{item-size-2} and~\eqref{item-remove-zero}.  For the former, we
simply check whether $|R| = 2$; if it is then $\{0\}$ is a maximal subsemigroup
of $R$.  For the latter, it suffices to check whether the matrix $P$ contains
the element $0$; if it does not, then $R\setminus \{0\}$ is a maximal
subsemigroup of $R$.

We require the following straightforward reformulation of
\eqref{item-remove-lambda} and~\eqref{item-remove-i}.

\begin{lem}\label{lem-isolated}
  Let $R = \M^{0}[I, G, \Lambda; P]$ be a finite regular Rees $0$-matrix
  semigroup over a group.
  
  If $|\Lambda|>1$ and $\lambda\in \Lambda$, then $M=(I\times G\times
  \Lambda\setminus\{\lambda\}) \cup \{0\}$ is a maximal subsemigroup of $R$ if
  and only if every vertex in the subgraph of the Graham-Houghton graph of $R$
  induced by $I\cup (\Lambda\setminus\{\lambda\})$ has at least one incident
  edge.

  Likewise, if $|I| > 1$ and $i\in I$, then $M = (I\setminus\{i\}\times
  G\times \Lambda)\cup \{0\}$ is a maximal subsemigroup of $R$ if and only if
  every vertex in the subgraph induced by $(I\setminus\{i\})\cup \Lambda$ has
  at least one incident edge.
\end{lem}
\begin{proof}
  If $M = (I\times G\times \Lambda\setminus\{\lambda\})\cup \{0\}$ is a maximal
  subsemigroup of $R$, then it is routine to verify that $M$ is not a maximal
  subsemigroup of
  type~\eqref{item-size-2},~\eqref{item-remove-zero},~\eqref{item-remove-i},~\eqref{item-rect},
  or~\eqref{item-subgroup}.  Hence $M$ is a maximal subsemigroup of
  type~\eqref{item-remove-lambda}, and there is at
  least one non-zero entry in each row and each column of ${(p_{\mu,
  i})}_{\mu\in\Lambda\setminus\{\lambda\}, i\in I}$, and so the subgraph of the
  Graham-Houghton graph of $R$ induced by $I\cup (\Lambda\setminus\{\lambda\})$
  has no vertex without an incident edge.

  If $\lambda\in\Lambda$ is such that the subgraph of the Graham-Houghton graph
  of $R$ induced by $I\cup (\Lambda\setminus\{\lambda\})$ has no vertex without
  an incident edge, then it is straightforward to check that  $M = (I\times
  G\times \Lambda\setminus\{\lambda\})\cup \{0\}$
  satisfies~\eqref{item-remove-lambda}, and hence is maximal.

  The proof of the second part is dual.
\end{proof}

It is straightforward to check whether the conditions in
Lemma~\ref{lem-isolated} hold, and so it is easy to find the maximal
subsemigroups of types \eqref{item-remove-lambda} and~\eqref{item-remove-i}.
Computing the maximal subsemigroups of type~\eqref{item-rect} is equivalent to
computing the maximal cliques in the complement of the Graham-Houghton graph of
$R$.  This problem is well-understood (and hard); see
\cite{Bron1973},~\cite{Tomita2006aa}, and~\cite{Moon1965}. 


\subsection{Maximal subsemigroups of type~\eqref{item-subgroup}}

It is more complicated to determine the maximal subsemigroups of
type~\eqref{item-subgroup} than it is to determine those of other types. We
require the following proposition. 

\begin{prop}[Section 4 in \cite{Graham1968ab}, Theorem 4.13.34 in
  \cite{Rhodes2009aa}] \label{prop-missing}
  Let $R$ be a regular finite Rees 0-matrix semigroup over a group $G$ with
  index sets $I$ and $\Lambda$, and let $n \in \mathbb{N}$ be the number of
  connected components of the Graham-Houghton graph of $R$. Then there exist
  non-empty subsets $I_{1}, \ldots, I_{n}$ of $I$ and $\Lambda_{1},
  \ldots, \Lambda_{n}$ of $\Lambda$ that partition $I$ and $\Lambda$,
  respectively, such that $I_{k} \cup \Lambda_{k}$ is a connected component of
  the Graham-Houghton graph for each $k$, and there exists a Rees 0-matrix
  semigroup  $R'=\M^0[I,G,\Lambda; P]$, where $P = {(p_{\lambda, i})}_{\lambda
  \in \Lambda, i \in I}$, which is isomorphic to $R$ and where the following hold:
  \begin{enumerate}[label=\emph{(\roman*)},ref=\roman*]
    \item 
      there are $i_{k} \in I_{k}$ and $\lambda_{k} \in \Lambda_{k}$
      such that $p_{\lambda_{k}, i_{k}} = 1_G$, for every $k$;
    
    \item 
      $\psi : R^{\prime} \setminus \{0\} \to G$ defined by
      $(i,g,\lambda)\mapsto g$ is an  isomorphism from
      $\{i\}\times G \times\{\lambda\}$ to $G$ whenever $p_{\lambda, i} = 1_G$;

    \item 
      $G_{k} = \big(\genset{E(R')} \cap (\{i_{k}\}\times G
      \times\{\lambda_{k}\})\big)\psi$ is a subgroup of $G$, for every $k$;
    
    \item 
      the matrix $P_k = {(p_{\lambda, i})}_{\lambda \in \Lambda_{k}, i
      \in I_{k}}$ consists of elements in $G_k \cup \{0\}$ and contains a
      generating set for $G_k$, for every $k$.
  \end{enumerate}
\end{prop}

We refer to a Rees 0-matrix semigroup satisfying the conditions of
Proposition~\ref{prop-missing} as being \textit{normalized}.  
Normalizing  a finite Rees 0-matrix semigroup $R = \M^0[I,G,\Lambda; P]$ 
consists of  finding the connected components of the Graham-Houghton graph of $R$
(time complexity $O(|I||\Lambda|)$), and multiplying every non-zero
entry in $P$ by some group element at most twice (complexity $O(|I||\Lambda|)$).
Hence such an $R$ can be normalized in time polynomial in $|I| |\Lambda|$.
For more information about Rees $0$-matrix semigroups and its connections to
graph theory, and normalization, we refer the reader to~\cite{Rhodes2009aa}.
For this point on, we suppose without loss of generality that
$R=\M^0[I,G,\Lambda;P]$ is a normalized regular finite Rees 0-matrix semigroup
where $G$ is a group, $I$ and $\Lambda$ are disjoint index sets, and $P =
{(p_{\lambda, i})}_{\lambda \in \Lambda, i \in I}$ is a $|\Lambda| \times |I|$
matrix over $G \cup \{ 0 \}$.

By~\cite[Theorem 2]{Graham1968ab}, each connected component of the
Graham-Houghton graph of $R$ corresponds to a regular Rees $0$-matrix semigroup
\begin{equation*}
  \left( I_{k} \times G_{k} \times \Lambda_{k} \right) \cup \{ 0 \}\
   =  \mathscr{M}^{0}[I_{k}, G_{k}, \Lambda_{k}; P_{k}],
\end{equation*}
and the idempotent generated subsemigroup $\genset{E(R)}$ of $R$ is the union
\begin{equation}\label{eq-idem}
    \bigcup\limits_{k = 1}^{n}
  \mathscr{M}^{0}[I_{k}, G_{k}, \Lambda_{k}; P_{k}].
\end{equation}

Throughout this section we let $S$ be a subsemigroup of $R$, and we use
the following notation:

\begin{itemize}
  \item 
    for $i \in I$ and $\lambda \in \Lambda$, let $H_{i, \lambda} = S \cap
    \left( \{ i \} \times G \times \{ \lambda \} \right)$ \emph{(the
    intersection of $S$ with an $\H$-class of $R$)};

  \item 
    for $k, l \in \{ 1, \ldots, n \}$, let $C_{k, l} = S \cap \left( I_{k}
    \times G \times \Lambda_{l} \right)$ \emph{(the intersection of $S$ with a
    block of $\H$-classes of $R$)}.
\end{itemize}

Observe that with this definition $C_{k, l} = \bigcup_{i \in I_{k},
\lambda \in \Lambda_{l}} H_{i, \lambda}$, and $S = \{0\} \cup \bigcup_{1 \leq
k, l \leq n} C_{k, l}$.

Note that if $S$ is regular, then $\H^{S} = \H^{R} \cap (S \times S)$
by~\cite[Proposition~2.4.2]{Howie1995aa}.  Therefore, if $S$ is regular, then
the non-zero $\H$-classes of $S$ are the non-empty sets of the form
$H_{i,\lambda}$.  In particular, if $S$ intersects each $\H$-class of $R$
non-trivially, then (since in this case $S$ is necessarily regular) it follows
that $H_{i,\lambda}$ is an $\H$-class of $S$ for all $i\in I$ and $\lambda\in
\Lambda$.

We next state three technical lemmas that describe the subsemigroups of $R$
that intersect every $\H$-class of $R$ non-trivially.

\begin{lem}\label{lem-block}
  The subsemigroup $S$ intersects every $\H$-class of $R$
  non-trivially if and only if $0\in S$ and there exists a subgroup $V$ of $G$
  and elements $g_1 = 1_G, g_{2}, \ldots, g_{n} \in G$ such that $G_{k}\leq
  g_{k}^{-1}Vg_{k}$ and $C_{k,l}=I_{k}\times g_{k}^{-1}Vg_{l}
  \times\Lambda_{l}$ for all $k,l \in \{1,\ldots,n\}$.
\end{lem}
\begin{proof}
  ($\Leftarrow$) The converse implication is immediate.

  ($\Rightarrow$)
  Recall that $p_{\lambda_k, i_k} = 1_G$ for all $k$, and so by Green's Lemma
  (Lemma~\ref{lem-green})
  $$H_{j, \mu} = H_{j,\lambda_{k}}(i_{k}, g, \mu) = (j, g, \lambda_k)
  H_{i_k,\mu},\qquad\text{and so}\qquad
  H_{j, i_{k}}H_{\lambda_k, \mu} = H_{j, \mu}$$
  for all $j\in I$, $\mu\in \Lambda$, and $g\in G$.

  Since $p_{\lambda_1, i_1} = 1_G$, it follows that the $\H$-class $H_{i_{1},
  \lambda_{1}}$ is a group, and that $H_{i_1,\lambda_1} = \{i_{1}\} \times V
  \times \{\lambda_{1}\}$ for some subgroup $V$ of $G$.
  By assumption, for each $k \in \{2,\ldots,n\}$ there exists an element
  $(i_{1},g_{k},\lambda_{k}) \in H_{i_{1},\lambda_{k}}$, and so   
  $$H_{i_{1}, \lambda_{k}} = H_{i_{1},\lambda_{1}}(i_{1}, g_{k},
  \lambda_{k}) = \{i_{1}\} \times V p_{i_{1},\lambda_{1}} g_{k} \times
  \{\lambda_{k}\} = \{i_{1}\} \times V g_{k} \times \{\lambda_{k}\}.$$
  Similarly, by assumption, for each $k\in\{2,\ldots,n\}$ the $\H$-class
  $H_{i_{k},\lambda_{1}}$ equals $\{i_{k}\} \times F_k \times \{\lambda_{1}\}$,
  where $F_k$ is some non-empty subset of $G$.  It follows that 
  $$H_{i_1,\lambda_1} = (i_{1}, g_{k}, \lambda_{k}) H_{i_k,\lambda_1} =
  \{i_1\}\times g_k F_k \times \{\lambda_1\},$$
  but as shown previously, $H_{i_1,\lambda_1} = \{i_{1}\}
  \times V \times \{\lambda_{1}\}$, and so $F_k = g_{k}^{-1} V$.
  Hence
  $$H_{i_{k},\lambda_{k}} = H_{i_{k},\lambda_{1}} H_{i_{1},\lambda_{k}} =
  \{i_{k}\}\times g_{k}^{-1} V g_{k} \times \{\lambda_{k}\}.$$

  Since $S$ is a finite semigroup that intersects every $\H$-class of $R$,
  it contains $E(R)$, and so it contains $\genset{E(R)}$.
  Therefore $H_{i_{k},\lambda_{k}}$ contains the elements
  $\genset{E(R)} \cap \left(\{i_{k}\}\times G \times \{\lambda_{k}\}\right)$,
  and so
  $$G_{k} = \big( \genset{E(R)} \cap \left(\{i_{k}\}\times G \times
  \{\lambda_{k}\}\right) \big)\psi \leq (H_{i_{k},\lambda_{k}})\psi
  = g_{k}^{-1}Vg_{k}$$
  for each $k$, as required.

  Let $k \in \{1, 2, \ldots, n\}$ and let $i \in I_{k} \setminus \{i_{k}\}$.
  Since the indices $i$ and $i_{k}$ are in the same connected component of the
  Graham-Houghton graph of $R$, there exists a path in the graph from $i$ to
  $i_{k}$.  That is, there exists an alternating sequence $(i = a_{1}, b_{1},
  a_{2}, b_{2}, \ldots, a_{m} = i_{k})$ of indices from $I_{k}$ and
  $\Lambda_{k}$, respectively, such that $p_{b_{j},a_{j}} \neq 0$ and
  $p_{b_{j},a_{j+1}} \neq 0$ for all possible $j$. Therefore $x =
  \prod_{j=1}^{m-1}(a_{j}, p_{b_{j},a_{j}}^{-1}, b_{j}) \in \genset{E(R)}
  \leq S$.
  It follows that
  $$x H_{i_{k}, \lambda_{k}} = \{i\} \times p_{b_{1},a_{1}}^{-1} 
  p_{b_{1},a_{2}} \cdots p_{b_{m-1},a_{m-1}}^{-1} p_{b_{m-1},a_{m}}
  g_{k}^{-1} V g_{k} \times \{\lambda_{k}\} \subseteq H_{i,\lambda_{k}}.$$
  By Proposition~\ref{prop-missing}(iv), this product of matrix entries is
  contained in the subgroup $G_{k} \leq g_{k}^{-1}Vg_{k}$. Therefore the
  expression simplifies to $x H_{i_{k}, \lambda_{k}} = \{i\} \times
  g_{k}^{-1}Vg_{k} \times \{\lambda_{k}\} \subseteq H_{i,\lambda_{k}}$. It
  follows that $(i, 1_{G}, \lambda_{k}) \in H_{i,\lambda_{k}}$ and so
  $$H_{i,\lambda_{k}} = (i,1_{G},\lambda_{k}) H_{i_{k},\lambda_{k}} = \{i\}
  \times g_{k}^{-1}Vg_{k} \times \{\lambda_{k}\}.$$
  By a similar argument $H_{i_{k},\lambda} = \{i_{k}\} \times g_{k}^{-1}Vg_{k}
  \times \{\lambda\}$ for all $\lambda \in \Lambda_{k}$. Hence for any $i
  \in I_{k}$, $\lambda \in \Lambda_{k}$
  $$H_{i,\lambda} = H_{i,\lambda_{k}} H_{i_{k},\lambda} = \{i\} \times
  g_{k}^{-1}Vg_{k}p_{\lambda_{k},i_{k}}g_{k}^{-1}Vg_{k} \times \{\lambda\} =
  \{i\} \times g_{k}^{-1} V g_{k} \times \{\lambda\},$$
  i.e. $C_{k,k} = I_{k} \times g_{k}^{-1}Vg_{k} \times \Lambda_{k}$.
  To conclude the proof let $k,l \in \{1,\ldots,n\}$ be arbitrary and let
  $H_{i,\lambda}$ be an $\H$-class of $S$ in $C_{k,l}$ (i.e. $i\in I_{k}$ and
  $\lambda\in\Lambda_{l}$).  
  Then we see that
  \begin{align}
    H_{i, \lambda}
    & = H_{i,\lambda_{k}}
        H_{i_{k},\lambda_{1}}
        H_{i_{1},\lambda_{l}}
        H_{i_{l},\lambda}
        &\label{eq-H-classes}\\
    & = ( \{i\} \times g_{k}^{-1} V g_{k} \times \{\lambda_{k}\} )
        ( \{i_{k}\} \times g_{k}^{-1} V \times \{\lambda_{1}\} )
        ( \{i_{1}\} \times V g_{l} \times \{\lambda_{l}\} )
        ( \{i_{l}\} \times g_{l}^{-1} V g_{l} \times \{\lambda\} )
        &\nonumber\\
    & = \{i\} \times
          g_{k}^{-1} V g_{k}
          p_{\lambda_{k}, i_{k}}
          g_{k}^{-1} V
          p_{\lambda_{1}, i_{1}}
          V g_{l}
          p_{\lambda_{l}, i_{l}}
          g_{l}^{-1} V g_{l}
        \times \{\lambda\}
        &\nonumber\\
    & = \{i\} \times g_{k}^{-1} V g_{l} \times \{\lambda\},
        &\nonumber
  \end{align}
  and so $C_{k,l}=I_{k} \times g_{k}^{-1} V g_{l} \times \Lambda_{l}$.
\end{proof}

\begin{lem}\label{lem-genset}
  Let $V$ be a subgroup of $G$ and let $g_{1} = 1_{G}, g_{2},\ldots,g_{n} \in
  G$ be elements such that $G_{k}\leq g_{k}^{-1}Vg_{k}$.
  Then $0\in S$ and $C_{k,l}=I_{k}\times g_{k}^{-1}Vg_{l} \times\Lambda_{l}$
  for all $k,l \in \{1,\ldots,n\}$
  if and only if
  $S = \genset{E(R),\{i_{1}\}\times V \times\{\lambda_{1}\},
  x_{2},\ldots,x_{n}, y_{2},\ldots,y_{n}}$, where
  $x_{k}=(i_{1},g_{k},\lambda_{k})$ and $y_{k}=(i_{k},g_{k}^{-1},\lambda_{1})$
  for all $k$.
\end{lem}

\begin{proof}
  ($\Rightarrow$)
  Let $x_{1} = y_{1} = (i_{1}, 1_{G}, \lambda_{1})$.
  Then since $x_{1}, y_{1} \in \{i_{1}\}\times V\times \{\lambda_{1}\}$, and
  since $\{i_{1}\}\times V\times\{\lambda_{1}\} = H_{i_{1},\lambda_{1}}$ by
  assumption, it suffices to show that the set $X = \left\{E(R),
  H_{i_{1},\lambda_{1}}, x_{1},\ldots,x_{n}, y_{1},\ldots,y_{n} \right\}$
  generates $S$. Clearly $\genset{X} \leq S$.
  
  Since $0 \in E(R) \leq \genset{X}$, to prove that $S \leq \genset{X}$, let
  $k,l \in \{1,\ldots,n\}$ be arbitrary and let $H_{i,\lambda}$ be an $\H$-class
  of $S$ in $C_{k,l}$ (i.e.~$i \in I_{k}$ and $\lambda\in\Lambda_{l}$). We must
  show that $H_{i,\lambda}\subseteq\genset{X}$.
  By Green's Lemma [Lemma~\ref{lem-green}],
  $H_{i_{k}, \lambda_{1}} = y_{k} H_{i_{1},\lambda_{1}} \subseteq \genset{X}$
  and
  $H_{i_{1}, \lambda_{l}} = H_{i_{1},\lambda_{1}} x_{l} \subseteq \genset{X}$.
  By~\eqref{eq-idem}, $(i, 1_{G}, \lambda_{k})$ and $(i_{l}, 1_{G}, \lambda)$
  are elements of $\genset{E(R)}$, and are therefore contained in $\genset{X}$.
  Hence by Green's Lemma, $H_{i,\lambda_{k}} = (i, 1_{G},
  \lambda_{k})H_{i_{k},\lambda_{k}} = (i, 1_{G}, \lambda_{k})y_{k}
  H_{i_{1},\lambda_{1}} x_{k} \subseteq \genset{X}$ and $H_{i_{l},\lambda} =
  H_{i_{l},\lambda_{l}}(i_{l}, 1_{G}, \lambda) = y_{l} H_{i_{1},\lambda_{1}}
  x_{l}(i_{l}, 1_{G}, \lambda) \subseteq \genset{X}$.
  By~\eqref{eq-H-classes},
  $H_{i, \lambda}
     = H_{i,\lambda_{k}}
       H_{i_{k},\lambda_{1}}
       H_{i_{1},\lambda_{l}}
       H_{i_{l},\lambda} \subseteq \genset{X}$,
  and it follows that $S \leq \genset{X}$.

  ($\Leftarrow$)
  That $S$ contains $0$, follows since $0\in E(R)$.

  To prove that $C_{k,l} \subseteq I_{k} \times g_{k}^{-1} V g_{l} \times
  \Lambda_{l}$ for all $k,l$, we must show that if $(i, g, \lambda) \in
  C_{k,l}$, for some $k,l$, then $g\in g_k ^ {-1} V g_l$. We proceed by
  induction on the length of a product in the generators equalling any given
  element.  Certainly, if $(i, g, \lambda)$ is one of the generators of $S$,
  then there exist $k, l$ such that $(i, g, \lambda) \in C_{k,l}$, and $g \in
  g_k^{-1} Vg_l$ by definition.

  Assume that if $(i, g, \lambda)\in S$  can be expressed as a product of
  length at most $m$ in the generators of $S$, then $g\in  g_k^{-1} V g_l$,
  where $k, l$ are such that $(i, g, \lambda)\in C_{k,l}$.  Let $(i, g,
  \lambda)\in S$ be any element that can be given as a product in $X$ of
  length $m + 1$. Then there exist $k, l$ such that $(i, g, \lambda)\in C_{k,
  l}$, and there exist $(i, h_1, \lambda'), (i', h_2, \lambda)\in S$ that can
  both be expressed as product of length at most $m$ over $X$ and $(i, h_1,
  \lambda')(i', h_2, \lambda) = (i, g, \lambda)$. Since this product is
  non-zero, it follows that $p_{\lambda', i'}\not =0$ and so $i'\in I_r$ and
  $\lambda'\in \Lambda_r$ for some $r$.  By induction, $h_1\in g_{k}^{-1}V
  g_{r}, h_2\in g_r^{-1}V g_l$, and so $$g = h_1p_{\lambda', i'} h_2\in
  g_{k}^{-1} V g_r p_{\lambda', i'} g_r^{-1} V g_l.$$ Since $p_{\lambda',i'}\in
  G_r\leq g_r^{-1}V g_r$ it follows that $g_r p_{\lambda', i'} g_r^{-1}\in V$.
  Therefore $g\in g_k^{-1} V g_l$, as required. 

  Let $k,l \in \{1,\ldots,n\}$.  To prove that $I_{k} \times g_{k}^{-1} V
  g_{l} \times \Lambda_{l} \subseteq C_{k,l}$ we will show that $\{i\}\times
  g_{k}^{-1}Vg_{l} \times\{\lambda\} \subseteq S$ for $i\in I_{k}$ and
  $\lambda\in \Lambda_{l}$.
  By~\eqref{eq-idem}, $(i,1_{G},\lambda_{k})$ and $(i_{l},1_{G},\lambda)$ are
  elements of $\genset{E(R)}$ and so
  \begin{equation*}
    \{i\} \times g_{k}^{-1}Vg_{l} \times \{\lambda\} = (i,1_{G},\lambda_{k})
    y_{k} \big( \{i_{1}\} \times V \times \{\lambda_{1} \}\big) x_{l}
    (i_{l},1_{G},\lambda)  \subseteq S.\qedhere
  \end{equation*}
\end{proof}

\begin{lem}\label{lem-genset-equiv}
  Let $V$ and $\overline{V}$ be subgroups of $G$ such that there exist $g_{1} =
  1_{G}, g_{2}, \ldots, g_{n}, \overline{g_{1}} = 1_{G},
  \overline{g_{2}},\ldots,\overline{g_{n}} \in G$ where $G_{k} \leq
  g_{k}^{-1}Vg_{k}$ and $G_{k} \leq
  \overline{g_{k}}^{-1}\overline{V}\overline{g_{k}}$ for all $k$.
  Let
  $X=\{E(R),\{i_{1}\}\times V \times\{\lambda_{1}\}, x_{2},\ldots,x_{n},
  y_{2},\ldots,y_{n}\}$
  and
  $\overline{X}=\{E(R),\{i_{1}\}\times \overline{V} \times\{\lambda_{1}\},
  \overline{x_{2}},\ldots,\overline{x_{n}},
  \overline{y_{2}},\ldots,\overline{y_{n}}\}$,
  where
  $x_{k}=(i_{1},g_{k},\lambda_{k})$,
  $\overline{x_{k}}=(i_{1},\overline{g_{k}},\lambda_{k})$,
  $y_{k}=(i_{k},g_{k}^{-1},\lambda_{1})$,
  and
  $\overline{y_{k}}=(i_{k},\overline{g_{k}}^{-1},\lambda_{1})$
  for all $k$.
  Then $\genset{X}=\genset{\overline{X}}$ if and only if
  $V = \overline{V}$ and
  $Vg_{k} = V\overline{g_{k}}$ for all $k$.
\end{lem}

\begin{proof}
  By Lemma~\ref{lem-genset},
  $\genset{X} = \genset{\overline{X}}$

  if and only if $\big( I_{k} \times g_{k}^{-1} V g_{l} \times \Lambda_{l} \big)
  = \big( I_{k} \times \overline{g_{k}}^{-1} \overline{V} \overline{g_{l}}
  \times \Lambda_{l} \big)$ for all $k, l \in \n$

  if and only if $g_{k}^{-1}Vg_{l} = {(\overline{g_{k}})}^{-1} \overline{V}
  \overline{g_{l}}$ for all $k,l \in \{1,\ldots,n\}$

  if and only if $V=\overline{V}$ and $g_{k}^{-1} V g_{l} =
  {(\overline{g_{k}})}^{-1} V \overline{g_{l}}$ for all $k, l \in \{ 1, \ldots,
  n \}$

  if and only if $V = \overline{V}$, $g_{k}^{-1} V g_{l} =
  {(\overline{g_{k}})}^{-1} V \overline{g_{l}}$, $g_{k}^{-1}V =
  {(\overline{g_{k}})}^{-1} V$, and $V g_{k} = V \overline{g_{k}}$ for all $k,l
  \in \{2,\ldots,n\}$

  if and only if $V=\overline{V}$ and $Vg_{k} = V\overline{g_{k}}$ for
  all $k \in \{2,\ldots,n\}$.
\end{proof}

Suppose that $V$ is a subgroup of $G$ such that $G_1\leq V$. By
Lemmas~\ref{lem-block},~\ref{lem-genset}, and~\ref{lem-genset-equiv}, in order
to find the subsemigroups $S$ of $R$ that intersect every $\H$-class of $R$
non-trivially, and such that $S \cap \big(\{i_1\}\times G\times
\{\lambda_1\}\big) = \{i_1\}\times V\times \{\lambda_1\}$, it suffices to find
an arbitrary set $\mathfrak{T}$ of representatives (called a
\textit{transversal}) of the right cosets of $V$ in $G$, and the sets
$$\set{g\in \mathfrak{T}}{G_k\leq g^{-1}Vg}$$
for all $k\geq 2$. More explicitly, given any transversal
$\mathfrak{T}$ of the right cosets of $V$ in $G$, by
Lemma~\ref{lem-genset-equiv} such subsemigroups of $R$ are in 1-1
correspondence with the Cartesian product
\begin{equation}\label{eq-coset-reps}
  \prod_{k = 2}^{n}\set{g\in \mathfrak{T}}{G_{k} \leq g^{-1}Vg}.
\end{equation}

Thus, if $t \in G$ is such that $G_{1} \leq t^{-1} V t$ and $\mathfrak{U}$ is an
arbitrary transversal of the right cosets of $t^{-1} V t$ in $G$, then the
subsemigroups $T$ of $R$ that intersect every $\H$-class of $R$ non-trivially,
and such that $T \cap \big(\{i_1\}\times G\times \{\lambda_1\}\big) =
\{i_1\}\times t^{-1}Vt\times \{\lambda_1\}$, are in 1-1 correspondence with the
Cartesian product
$$\prod_{k = 2}^{n} \bigset{g\in \mathfrak{U}}{G_{k} \leq g^{-1}
\left( t^{-1} V t \right)g}.$$
Note that we call any such subsemigroup $T$ a \emph{subsemigroup arising
from a conjugate of $V$}.  However, if $\mathfrak{T}$ is a transversal of the
right cosets of $V$ in $G$, then the set $\bigset{t^{-1} g}{g \in \mathfrak{T}}$
is a transversal of the right cosets of $t^{-1}Vt$ in $G$.  In particular, given
any transversal $\mathfrak{T}$ of the right cosets of $V$ in $G$, the collection
of all such subsemigroups $T$ is in 1-1 correspondence with the Cartesian
product $$\prod_{k = 2}^{n} \bigset{t^{-1}g}{g \in \mathfrak{T}\ \text{and}\
G_{k} \leq (t^{-1}g)^{-1} \left(t^{-1} V t \right)(t^{-1}g)} = \prod_{k = 2}^{n}
\bigset{t^{-1}g}{g \in \mathfrak{T}\ \text{and}\ G_{k} \leq g^{-1}Vg}.$$
Therefore, to find all subsemigroups arising from conjugates of $V$, 
it suffices to find the conjugates of $V$, as well as the set
in~\eqref{eq-coset-reps}.

We require the following lemma for the proof of
Proposition~\ref{prop-genset-rms-6}.

\begin{lem}\label{lem-super}
  If $S$ intersects every $\H$-class of $R$ non-trivially, then every
  non-zero $\H$-class of $S$ contains the same number of elements. In
  particular, if $S$ intersects every $\H$-class of $R$ non-trivially
  and contains a non-zero $\H$-class of $R$, then $S=R$.
\end{lem}

\begin{proof}
  By Lemma~\ref{lem-block}, there exists a subgroup $V$ of $G$ such that every
  non-zero $\H$-class of $S$ has the form $\{i\}\times g_{k}^{-1}Vg_{l}
  \times\{\lambda\}$ for some indices $i,\lambda$ and for some elements
  $g_{k},g_{l}\in G$. The result follows.
\end{proof}

The following proposition will allow us to find all of the maximal
subsemigroups in $R$ of type~\eqref{item-subgroup} arising from conjugates of a
given maximal subgroup of the group $G$. 

\begin{prop}\label{prop-genset-rms-6}
  Let $R = \M^0[I,G,\Lambda;P]$ be a normalized regular finite Rees 0-matrix
  semigroup over a group $G$, where $I$ and $\Lambda$ are disjoint index sets,
  and $P = {(p_{\lambda, i})}_{\lambda \in \Lambda, i \in I}$ is a $|\Lambda|
  \times |I|$ matrix over $G \cup \{ 0 \}$, and let $S$ be a subsemigroup of
  $R$. Suppose that the number of connected components of the Graham-Houghton
  graph of $R$ is $n$, and let $G_1, \ldots, G_n$ and $i_1, \ldots, i_n\in I$
  and $\lambda_1, \ldots, \lambda_n\in \Lambda$ be as defined in
  Proposition~\ref{prop-missing}.
  Then $S$ is a maximal subsemigroup of $R$ that intersects every
  $\H$-class of $R$ non-trivially if and only if there exists a
  maximal subgroup $V$ of $G$  and elements $g_1 = 1_G,
  g_2,\ldots, g_n\in G$ such that $G_k \leq g_{k}^{-1} V g_{k}$ for all $k$,
  and 
  $$ S = \genset{E(R), \{i_1\}\times V \times \{\lambda_1\},  x_{2}, \ldots, x_{n},
  y_{2}, \ldots, y_{n}}$$ 
  where $x_{k} = (i_{1}, g_{k}, \lambda_{k})$ and $y_{k} = (i_{k}, g_{k}^{-1},
  \lambda_{1})$ for all $k$.
\end{prop}

\begin{proof}
  ($\Rightarrow$)
  After applying the direct implications of Lemmas~\ref{lem-block}
  and~\ref{lem-genset}, it remains to prove that $V$ is a maximal subgroup of
  $G$.
  Let $K$ be a subgroup of $G$ with $V \leq K \leq G$.  Then $T=\genset{E(R),
  \{i_{1}\}\times K \times\{\lambda_{1}\}, x_{2},\ldots,x_{n},
  y_{2},\ldots,y_{n}}$ is a subsemigroup of $R$, and $S\leq T$.
  Since $S$ is maximal, either $T=S$ or $T=R$.  In the former case, $K=V$ by
  Lemma~\ref{lem-genset-equiv}; in the latter case, by the converse implication
  of Lemma~\ref{lem-genset}, it follows that $\{i_{1}\}\times K
  \times\{\lambda_{1}\} = \{i_{1}\}\times G \times\{\lambda_{1}\}$, i.e. $K=G$.
  
  ($\Leftarrow$) 
  By the converse implication of Lemma~\ref{lem-genset}, it follows that $0
  \in S$ and $C_{k,l}=I_{k}\times g_{k}^{-1}Vg_{l} \times\Lambda_{l}$ for all
  $k,l$; in particular, $S$ is a proper subsemigroup of $R$ that intersects
  every $\H$-class of $R$ non-trivially.
  
  It remains to show that $S$ is maximal.  Let $x = (i, g, \lambda)\in
  R \setminus S$ be arbitrary. Then $\genset{S,x}$ is a subsemigroup of $R$
  that intersects every $\H$-class of $R$ non-trivially.  Since
  $\genset{S,x} \cap \left(\{i\}\times G \times\{\lambda\}\right)$ contains at
  least $|V|+1$ elements, so does the group $\genset{S,x} \cap
  \left(\{i_{1}\}\times G \times\{\lambda_{1}\}\right)$ by
  Lemma~\ref{lem-super}.  Therefore
  $$V = \big(S \cap \left(\{i_{1}\}\times G
  \times\{\lambda_{1}\}\right)\big)\psi \lneq \big(\genset{S,x} \cap
  \left(\{i_{1}\}\times G \times\{\lambda_{1}\}\right)\big)\psi \leq G.$$

  Since $V$ is a maximal subgroup of $G$, $\big(\genset{S,x} \cap \left(
  \{i_{1}\} \times G \times \{\lambda_{1}\}\right)\big)\psi = G$.
  The result follows by Lemma~\ref{lem-super}.
\end{proof}

Let $V$ be a representative of a conjugacy class of maximal subgroups of $G$.
By Lemma~\ref{lem-genset-equiv}, Proposition~\ref{prop-genset-rms-6}, and the
previous arguments, to find all of the maximal subsemigroups in $R$ of
type~\eqref{item-subgroup} that arise from conjugates of $V$, we first require
any transversal $\mathfrak{T}$ of the right cosets of $V$ in $G$, and any
transversal $\mathfrak{U}$ of the right cosets of $N_{G}(V)$ in $G$, where
$N_{G}(V)$ is the \emph{normalizer} of $V$ in $G$. Note that since $N_{G}(V)$ is
subgroup of $G$ that contains the maximal subgroup $V$, it follows that either
$N_{G}(V) = V$, in which case we may choose $\mathfrak{U} = \mathfrak{T}$, or
$N_{G}(V) = G$, in which case we may choose $\mathfrak{U} = \{1_{G}\}$. We then
require all coset representatives $g \in \mathfrak{T}$ such that $G_{k} \leq
g^{-1} V g$, for all $k \in \{2, \ldots, n\}$, and all coset representatives $t
\in \mathfrak{U}$ such that $G_{1} \leq t^{-1} V t$.
By performing this process for each conjugacy class of maximal subgroups of $G$,
we find all maximal subsemigroups of $R$ of type~\eqref{item-subgroup}.

\begin{cor}\label{cor-number-type-6}
  Let $V$ be a maximal subgroup of $G$, let $\mathfrak{T}$ be an arbitrary
  transversal of the right cosets of $V$ in $G$ and let $\mathfrak{U}$ be an
  arbitrary transversal of the right cosets of $N_{G}(V)$ in $G$.  Then the
  number of maximal subsemigroups in $R$ of type~\eqref{item-subgroup} that
  arise from conjugates of $V$ is
  $$M = |\set{t \in \mathfrak{U}}{G_{1} \leq t ^{-1} V t}|
  \cdot
  \prod_{k = 2}^{n}|\set{g \in \mathfrak{T}}{G_{k} \leq g^{-1}Vg}|.$$
  If $N_{G}(V) = G$, then $M \leq {[G : V]} ^ {n - 1}$. Otherwise, $N_{G}(V) =
  V$, and so $M \leq {[G : V]} ^ n$.
\end{cor}

The upper bounds in Corollary~\ref{cor-number-type-6} are tight. Let $B(G, m)$
denote the Brandt semigroup $\mathscr{M}^0[I, G, I; P]$ where $P$ is the
identity matrix and $|I| = m$. For example, suppose that $G = S_{3}$, the
symmetric group of degree $3$. The number of maximal subsemigroups of $B(G, m)$
of type~\eqref{item-subgroup} arising from conjugates of $V = A_{3}$, the
alternating group of degree $3$, is ${[G : V]} ^ {m - 1} = 2 ^ {m - 1}$, whilst
the number arising from $V = \genset{(1\ 2)}$, a cyclic group of order $2$, is
${[G : V]} ^ {m} = 3 ^ m$.

On the other hand, suppose that $G = \{1_{G}, x\} = \genset{x}$ is a cyclic
group of order $2$, that $|I| = |\Lambda| = 2$, and that
\begin{equation*}
  P = 
  \begin{pmatrix}
    1_G & 1_G\\
    1_G & x
  \end{pmatrix}.
\end{equation*} 
Then in $R = \mathscr{M}^0[I, G, \Lambda; P]$ the subgroup $G_{1}$ is equal to
$G$, and so there are no maximal subsemigroups of type~\eqref{item-subgroup},
since no maximal subgroup of $G$ contains $G_{1}$. 

A method for finding the maximal subsemigroups of
type~\eqref{item-subgroup} is given in Algorithm~\ref{algorithm-type-6}.

\begin{algorithm}
  \caption{Maximal subsemigroups of
  type~\eqref{item-subgroup}}\label{algorithm-type-6}
  \begin{algorithmic}[1]

    \item[\textbf{Input:}] $R$, a finite
      regular Rees 0-matrix semigroup over a group. 

    \item[\textbf{Output:}] the maximal subsemigroups $\mathfrak{M}$ of $R$ of
      type~\eqref{item-subgroup}.

    \item 
      find an isomorphism $\Psi:R \to \mathscr{M}^0[I, G, \Lambda; P] = R'$ such
      that $R'$ is normalized
      \Comment{Section 4 of~\cite{Graham1968ab}}

    \item
      find partitions $I_{1},\ldots,I_{n}$ and $\Lambda_{1},\ldots,\Lambda_{n}$
      of $I$ and $\Lambda$ such that $I_{k}\cup\Lambda_{k}$ is a connected
      component of the Graham-Houghton graph of $R'$ for all $k$
      \Comment{Standard breadth or depth first
      search~\cite[Section 4.1]{Sedgewick2011aa}}
    
    \item
      fix $i_1\in I_1$, \ldots, $i_n \in I_n$,  and $\lambda_1 \in \Lambda_1$,
      \ldots, $\lambda_n \in \Lambda_n$ such that $p_{\lambda_k, i_k} = 1_G$
      for all $k$ 

    \item $E(R') := \set{(i, p_{\lambda, i} ^ {-1},
      \lambda)}{i\in I,\ \lambda\in\Lambda,\ p_{\lambda i} \not= 0} \cup \{0\}$
      \Comment{The idempotents of $R'$}

    \item set $G_k := \genset{\set{p_{\lambda, i}}{i \in I_k,\ \lambda\in
      \Lambda_k, p_{\lambda, i} \not= 0}}$ for all $k$
      \Comment{Proposition~\ref{prop-missing}(iv)}
    
    \item set $\mathcal{C}$ to be a set of conjugacy class representatives of
      the maximal subgroups of $G$ 
      \Comment{Standard group theoretic
      algorithm~\cite{Cannon2004aa,Eick2001aa}}

    \item $\mathfrak{M} := \varnothing$
      
    \For{$V\in \mathcal{C}$}

    \State{compute $\mathfrak{U}$, a transversal of the right cosets of
    $N_{G}(V)$ in $G$}

      \State{$T_1 := \varnothing$}

      \For{$t\in \mathfrak{U}$}

        \If{$G_1\leq t^{-1} V t$}

        \State{$T_1\gets T_1 \cup \{t\}$}
 
        \EndIf{}

      \EndFor{}

      \State{compute $\mathfrak{T}$, a transversal of the right cosets of $V$
      in $G$}
      
      \For{$k\in \{2, \ldots, n\}$}

        \State{$T_k := \varnothing$}

        \For{$g\in \mathfrak{T}$}

          \If{$G_k\leq g^{-1} V g$}

            \State{$T_k\gets T_k \cup \{g\}$}
   
          \EndIf{}

        \EndFor{}
      \EndFor{}

      \For{$t_1\in T_1$, $t_2\in T_2$, \ldots, $t_n\in T_n$}

      \State$\mathfrak{M}\gets \mathfrak{M} \cup \big(\{\{E(R'),\ 
          \{i_1\}\times t_1^{-1}Vt_1 \times \{\lambda_1\},\ 
          (i_1, t_1^{-1}t_2, \lambda_2),\ 
          \ldots,$

          \qquad\qquad\qquad$(i_1, t_1^{-1}t_n, \lambda_n),\ 
          (i_2, t_2^{-1}t_1, \lambda_1),\ 
          \ldots,\ 
          (i_n, t_n^{-1}t_1, \lambda_1)\}\}\big)\Psi^{-1}$
        
      \EndFor{}

    \EndFor{}

    \State{\Return{$\mathfrak{M}$}}
  \end{algorithmic}
\end{algorithm}


\subsection{An example}

In this section, we give an example of a Rees $0$-matrix semigroup and show
how our algorithm can be applied to calculate its maximal subsemigroups.

Let $R$ be the regular Rees $0$-matrix semigroup $\mathscr{M} ^ {0} [I, S_{4},
\Lambda; P]$, where $I = \{1, \ldots, 6\}$, $\Lambda = \{-6, \ldots, -1\}$,
$S_{4}$ is the symmetric group of degree $4$, and $P$ is the $6 \times
6$~matrix 
\begin{equation*}
P = 
\begin{pmatrix}
  (3\ 4) & (1\ 3\ 2\ 4) & (1\ 4)(2\ 3)  & 0         & 0            & 0       \\
  (2\ 4) & 0            & (1\ 3\ 2)     & 0         & 0            & 0       \\
  0      & (3\ 4)       & 0             & 0         & 0            & 0       \\
  0      & 0            & 0             & (1\ 4\ 3) & (1\ 3)(2\ 4) & 0       \\
  0      & 0            & 0             & (1\ 4)    & (1\ 4\ 2)    & 0       \\
  0      & 0            & 0             & 0         & 0            & (1\ 4\ 2)
\end{pmatrix}.
\end{equation*}
A diagram of the Graham-Houghton graph of $R$ is shown in
Figure~\ref{fig-graham-houghton}. 

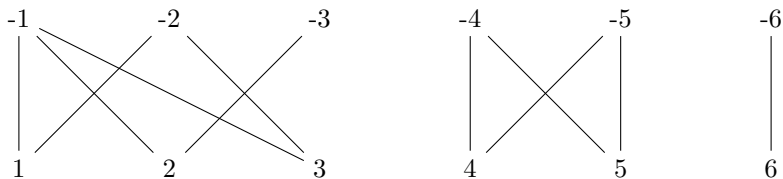
\begin{figure}[!hb]
\begin{center}
\begin{tikzpicture}
  \node (1) at (0,  2) {-1};
  \node (2) at (2,  2) {-2};
  \node (3) at (4,  2) {-3};
  \node (4) at (6,  2) {-4};
  \node (5) at (8,  2) {-5};
  \node (6) at (10, 2) {-6};

  \node (a) at (0,  0) {1};
  \node (b) at (2,  0) {2};
  \node (c) at (4,  0) {3};
  \node (d) at (6,  0) {4};
  \node (e) at (8,  0) {5};
  \node (f) at (10, 0) {6};

  \edge{1}{a};
  \edge{1}{b};
  \edge{1}{c};
  \edge{2}{a};
  \edge{2}{c};
  \edge{3}{b};

  \edge{4}{d};
  \edge{4}{e};
  \edge{5}{d};
  \edge{5}{e};
  
  \edge{6}{f};

\end{tikzpicture}
\end{center}
\caption{The Graham-Houghton graph $\Gamma(R)$ of $R$.}
\label{fig-graham-houghton}
\end{figure}

Since $|R| = 6 ^ 2 \cdot 4! + 1 = 865$, $R$ has no maximal subsemigroup
of type~\eqref{item-size-2}. Since $P$ contains $0$, $R\setminus \{0\}$ is not
a subsemigroup and so there are no maximal subsemigroups of type
\eqref{item-remove-zero} either.

The removal of any of the vertices $2$, $6$, or $-6$ from the Graham-Houghton
graph $\Gamma(R)$, results in an induced subgraph containing a vertex without any
incident edges.  Hence the only maximal subsemigroups of $R$ of
types~\eqref{item-remove-lambda} and \eqref{item-remove-i} arise from removing
any of the vertices $-1, -2, -3, -4$, or $-5$ or $1, 3, 4$, or $5$,
respectively.  Thus there are $9$ subsemigroups of
types~\eqref{item-remove-lambda} and~\eqref{item-remove-i}.

The maximal cliques in the dual of this graph
corresponding to maximal subsemigroups are: 
$$
\{ -1, -2, -3, -4, -5, 6 \}, 
\{ -1, -2, -3, -6, 4, 5 \}, 
\{ -1, -2, -3, 4, 5, 6 \}, 
\{ -2, -4, -5, -6, 2 \},
$$
$$
\{ -2, -4, -5, 2, 6 \},
\{ -2, -6, 2, 4, 5 \},
\{ -2, 2, 4, 5, 6 \},
\{ -3, -4, -5, -6, 1, 3 \},
\{ -3, -4, -5, 1, 3, 6 \},
$$
$$
\{ -3, -6, 1, 3, 4, 5 \},
\{ -3, 1, 3, 4, 5, 6 \},
\{ -4, -5, -6, 1, 2, 3 \},
\{ -4, -5, 1, 2, 3, 6 \},
\{ -6, 1, 2, 3, 4, 5 \}.
$$
The only maximal cliques in the dual of $\Gamma(R)$ that do not correspond to
maximal subsemigroups are $\{-6, \ldots, -1\}$ and $\{1, \ldots, 6\}$.  Hence
there are $14$ maximal subsemigroups of $R$ of type \eqref{item-rect}.

To calculate the maximal subsemigroups of $R$ of type~\eqref{item-subgroup},
it is first necessary to find a normalized Rees 0-matrix semigroup that is
isomorphic to $R$. One such normalized Rees 0-matrix semigroup is
$R' = \mathscr{M} ^ {0} [I, S_{4}, \Lambda; P']$,
where
\begin{equation*}
P' = \left(
\begin{array}{cccccc}
  \id   & \id       & \id           & 0       & 0            & 0       \\
  \id   & 0         & (1\ 2)(3\ 4)  & 0       & 0            & 0       \\
  0     & \id       & 0             & 0       & 0            & 0       \\
  0     & 0         & 0             & \id     & \id          & 0       \\
  0     & 0         & 0             & \id     & (1\ 2\ 3\ 4) & 0       \\
  0     & 0         & 0             & 0       & 0            & \id
\end{array}
\right)
\end{equation*}
and where $\id$ is the identity permutation on $\{ 1, 2, 3, 4 \}$.

It is clear by inspecting $\Gamma(R)$ that $R$, and hence $R'$, have $3$
connected components.

To find the maximal subsemigroups in $R$ of type \eqref{item-subgroup}, we
first find the groups $G_{1}, G_{2}, G_{3}$ from the definition of a normalized
Rees 0-matrix semigroup.  The groups $G_{1}, G_{2}$, and $G_{3}$ are generated
by the non-zero matrix entries corresponding to the relevant connected component
of the graph $\Gamma(R)$, and so: $$G_{1} = \genset{(1\ 2)(3\ 4)},\quad G_{2} =
\genset{(1\ 2\ 3\ 4)},\quad \text{and} \quad G_{3} = \mathbf{1},$$ where
$\mathbf{1}$ is the trivial subgroup of $S_{4}$.

We then determine the maximal subgroups $V$ of $G$ up to conjugacy, and for
each $V$, find the maximal subsemigroups of $R$ that arise from conjugates of
$V$. For each such $V$, it suffices to find a transversal $\mathfrak{U}$ of the
right cosets of $N_{G}(V)$ in $G$ and a transversal $\mathfrak{T}$ of the right
cosets of $V$ in $G$, along with the sets
$$\set{g \in \mathfrak{U}}{G_{1} \leq g^{-1}Vg}\quad\text{and}\quad\set{g\in
\mathfrak{T}}{G_{k} \leq g^{-1}Vg},\quad\text{for all}\ k\in \{2, 3\}.$$
If, for a given choice of $V$, any of these sets is empty, then there are no
maximal subsemigroups of $R$ arising from conjugates of $V$. 

Up to conjugacy, there are $3$ maximal subgroups of $S_{4}$: the alternating
group $A_{4}$ in its natural representation; the symmetric group $S_{3}$; and
the dihedral group of order $8$, $D_{4} = \genset{(1\ 2), (1\ 3)(2\ 4)}$. We
consider each of these cases separately.

\textbf{Case 1: $V = A_{4}$.} 
Since $G_{2}$ is not a subgroup of $g^{-1}A_{4}g = A_{4}$ for any $g\in G$,
there are no maximal subsemigroups of $R$ of type~\eqref{item-subgroup} arising
from conjugates of $A_4$.

\textbf{Case 2: $V = S_{3}$.} 
Since $G_{1}$ has no fixed points but every subgroup of every conjugate of
$S_{3}$ fixes at least one point, it follows that $G_{1}\not\leq g^{-1}S_{3}g$
for any $g\in G$.  Hence there are no maximal subsemigroups of $R$ arising from
conjugates of $S_{3}$.

\textbf{Case 3: $V = D_{4}$.} 
Since $D_{4}$ is not a normal subgroup of $S_{4}$, it
follows that $N_{S_{4}}(D_{4}) = D_{4}$. We choose the transversal:
$$\mathfrak{T} = \mathfrak{U} = \{ \id,\ (2\ 3),\ (2\ 4\ 3) \}.$$
All $3$ conjugates of $D_{4}$ contain $G_{1}$, the only conjugate of $D_{4}$
that contains $G_{2}$ is ${(2\ 3)} ^ {-1} D_{4} (2\ 3)$, and every subgroup of
$S_{4}$ contains $G_{3}$.  Hence there are $9$ maximal subsemigroups of $R'$
(and hence $R$) of type~\eqref{item-subgroup} that arise from conjugates of
$D_{4}$.

In total there are $32$ maximal subsemigroups of $R$.


\section{Arbitrary semigroups}\label{section-arbitrary}

In this section, we consider the problem of computing the maximal subsemigroups
of an arbitrary finite semigroup, building on the results of
Section~\ref{section-rees}. 

By Lemma~\ref{lem-graham}, if $M$ is a maximal subsemigroup of a finite
semigroup $S$, then $S\setminus M\subseteq J$ for some $\J$-class $J$
of $S$. By Proposition~\ref{prop-graham}, either $J$ is non-regular and $M$ has
the following form:

\begin{enumerate}[label=(S\arabic*), ref=S\arabic*]
  \item\label{item-nonreg}
    $M = S \setminus J$
    (Proposition~\ref{prop-graham}\eqref{item-prop-nonreg}),
\end{enumerate}
or $J$ is regular and $M$ has precisely one of the following forms:
\begin{enumerate}[label=(S\arabic*), ref=S\arabic*, resume]
  \item\label{item-every}
    $M\cap J$ has non-empty intersection with every
    $\H$-class in $J$
    (Proposition~\ref{prop-graham}\eqref{item-prop-intersect});

  \item\label{item-both}
    $M \cap J$ is a non-empty union of both $\L$- and $\R$-classes of $S$
    (Proposition~\ref{prop-graham}(c)\eqref{item-prop-rect});

  \item\label{item-l}
   $M\cap J$ is a non-empty union of $\L$-classes of $S$
   (Proposition~\ref{prop-graham}(c)\eqref{item-prop-cols});

  \item\label{item-r}
   $M\cap J$ is a non-empty union of $\R$-classes of $S$
   (Proposition~\ref{prop-graham}(c)\eqref{item-prop-rows});
  
  \item\label{item-empty}
    $M\cap J = \varnothing$
    (Proposition~\ref{prop-graham}(c)\eqref{item-prop-empty}).
\end{enumerate}

Throughout this section, $S$ denotes an arbitrary finite semigroup with
generating set $X$, $J$ denotes a $\J$-class of $S$, $\phi: J^* \to
R = \mathscr{M}^{0}[I, G, \Lambda; P]$ is an isomorphism from the principal factor
$J^{*}$ of $J$, and $X'$ consists of those generators $x\in X$ such
that $J_{x} > J$, where $J_x$ is the $\J$-class of $x$.

We will use the following straightforward lemma repeatedly in this section, for
which we include a proof for completeness.

\begin{lem}\label{lem-subsemigroup}
  Let $T$ be a subset of $S$ such that $S \setminus T \subseteq J$.  Then $T$
  is a subsemigroup of $S$ if and only if
  \begin{enumerate}[label=\emph{(\roman*)},ref=\roman*]

    \item\label{item-assumption-1}
      $\genset{X ^ {\prime}} \subseteq T$;

    \item\label{item-assumption-2}
      if $x, y \in T \cap J$, then $xy \in J$ implies that $xy \in T$; and

    \item\label{item-assumption-3}
      if $x \in T \cap J$ and $y \in \genset{X'}$, then $xy \in J$ implies that
      $xy \in T$, and $yx \in J$ implies that $yx \in T$.

  \end{enumerate}
\end{lem}
\begin{proof}
  ($\Rightarrow$)
  The second condition holds since $T$ is a subsemigroup,
  and since $T$ additionally contains $X^{\prime}$, it follows that the first
  and third conditions hold.

  ($\Leftarrow$)
  Suppose that $x, y \in T$ and that $xy \in J$.  Then $x, y \in J \cup
  \genset{X ^ {\prime}}$.
  If $x, y \in J$ then $xy \in T$ by~\eqref{item-assumption-2}.
  If $x, y \in \genset{X ^ {\prime}}$, then $xy \in \genset{X ^ {\prime}}
  \subseteq T$. For the remaining cases, $xy \in T$
  by~\eqref{item-assumption-3}.
\end{proof}


\subsection{Maximal subsemigroups arising from non-regular $\J$-classes:
type~\eqref{item-nonreg}}\label{sec-nonreg}

In this section, we characterise those maximal subsemigroups of a finite
semigroup $S$ arising from the exclusion of a non-regular $\J$-class.
Throughout this section we suppose that $J$ is a non-regular
$\J$-class of $S$.

\begin{lem}\label{lem-semigroup-remove-non-regular}
  Either $\genset{S \setminus J} = S$ or $S \setminus J$ is a maximal
  subsemigroup of $S$.
\end{lem}

\begin{proof}
   Clearly, $S \setminus J \subseteq \genset{S \setminus J} \subseteq S$.
   If $\genset{S \setminus J} \neq S$, then $\genset{S \setminus J}$ is a proper
   subsemigroup of $S$, and is therefore contained in a maximal subsemigroup of
   $S$ arising from $J$.  However, by Proposition~\ref{prop-graham}, the only
   maximal subsemigroup of $S$ that can arise from a non-regular $\J$-class is
   formed by removing it. Therefore $S \setminus J$ is a maximal subsemigroup of
   $S$.  
\end{proof}

\begin{prop}\label{prop-non-regular-arbitrary}
  Let $S$ be a finite semigroup generated by $X$, let $J$ be a non-regular
  $\J$-class of $S$, and let $X ^
  {\prime} \subseteq X$ be the set of generators whose $\J$-classes
  are strictly greater than $J$ in the $\J$-class partial order on
  $S$.  Then $S \setminus J$ is a maximal subsemigroup of $S$ if and only if
  $\genset{X'}\cap J = \varnothing$.
\end{prop}

\begin{proof}
  ($\Rightarrow$) 
  Since $\genset{X ^ {\prime}} \leq \genset{S \setminus J} = S \setminus J$, it
  follows that $\genset{X ^ {\prime}}\cap J = \varnothing$.

  ($\Leftarrow$) 
  We prove the contrapositive. If $S\setminus J$ is not a maximal subsemigroup,
  then, by Lemma~\ref{lem-semigroup-remove-non-regular}, $S\setminus J$ is a
  generating set for $S$. Hence $\genset{X'} \cap J = J\not=\varnothing$.
\end{proof}

By Proposition~\ref{prop-non-regular-arbitrary}, if $J$ is a maximal
non-regular $\J$-class of $S$, then $S \setminus J$ is a maximal
subsemigroup of $S$.

\begin{cor}
  Let $J$ be a non-regular $\J$-class of a finite semigroup $S$
  generated by $X$ where $x\not\in \genset{X\setminus\{x\}}$ for all $x\in X$.
  Then $S \setminus J$ is not a maximal subsemigroup of $S$ if and only if 
  $J \cap X = \varnothing$.
\end{cor}

\begin{proof}
  $(\Leftarrow)$ 
  If $J \cap X = \varnothing$, then $S = \genset{X} \subseteq \genset{S
  \setminus J} \subseteq S$, i.e.\ $\genset{S \setminus J} = S$. 

  $(\Rightarrow)$ 
  Since $S\setminus J$ is not maximal,
  Lemma~\ref{lem-semigroup-remove-non-regular} implies that $S\setminus J$
  generates $S$. Thus $\genset{X'}\cap J = J$ and so if $x\in J\cap X$, then
  $x\in \genset{X'} \leq \genset{X\setminus\{x\}}$, which contradicts the
  assumption on $X$. Therefore $J\cap X = \varnothing$. 
\end{proof}


\subsection{Maximal subsemigroups arising from regular $\J$-classes that
intersect every $\H$-class:
type~\eqref{item-every}}\label{subsec-intersect-every-H-class}

In this section, we consider those maximal subsemigroups of a finite semigroup
$S$ that arise from the exclusion of elements in a regular $\J$-class $J$ of
$S$, and that intersect every $\H$-class of $S$ non-trivially.  In other words,
we are considering maximal subsemigroups of type~\eqref{item-every}.
Since the principal factor $J^{*}$ is isomorphic to a regular Rees 0-matrix
semigroup over a group, the algorithms described in Section~\ref{section-rees}
could be used to compute the maximal subsemigroups of $J ^ *$. 
The purpose of this section is to characterise the maximal subsemigroups of
type~\eqref{item-every} in terms of the maximal subsemigroups of $J^{*}$.

\begin{lem}\label{lem-intersect}
  Let $S$ be a finite semigroup, and let $T$ be a subset of $S$ such that $S
  \setminus T$ is contained in a regular $\J$-class $J$ of $S$. Suppose that $T$
  intersects every $\H$-class of $S$ and let $E$ be a set consisting of one
  idempotent from each $\L$-class of $J$.  Then $T$ is a subsemigroup of $S$ if
  and only if 
  \begin{enumerate}[label=\emph{(\roman*)},ref=\roman*]
    \item 
      $EX' \subseteq T$;
    \item 
      if $x, y \in T \cap J$, then $xy \in J$ implies $xy \in T$
      (Lemma~\emph{\ref{lem-subsemigroup}}\eqref{item-assumption-2}).
    \end{enumerate}
  
\end{lem}

\begin{proof}
  $(\Rightarrow)$
  Since $T$ is a subsemigroup,
  Lemma~\ref{lem-subsemigroup}\eqref{item-assumption-2} holds.  Since $T$ is
  finite and intersects every $\H$-class of $S$, it contains every idempotent of
  $S$. By assumption, $X'\subseteq T$, and so $EX' \subseteq T$.

  $(\Leftarrow)$
  It suffices to show that the remaining conditions of
  Lemma~\ref{lem-subsemigroup} hold.  In order to do this, we first show that
  $E\genset{X^{\prime}} \subseteq T$.  Let $x \in E\genset{X ^ \prime} \cap J$.
  By definition, there exists an idempotent $e_{1} \in E$ and a sequence of
  generators $x_{1}, \ldots, x_{n} \in X^{\prime}$ such that $x = e_{1}x_{1}
  \cdots x_{n}$.  Since the element $e_{1}$ and the product $e_{1}x_{1}\cdots
  x_{n}$ are both members of $J$, it follows that the intermediate product
  $e_{1}x_{1}\cdots x_{k}$ is in $J$ for every $k \in \{1,\ldots,n\}$.  Hence by
  definition of the set $E$, for each index $k < n$ there exists an idempotent
  $e_{k+1} \in E$ such that $e_{k+1} \L e_1x_{1} \cdots x_{k}$, and in
  particular $(e_{1}x_{1}\cdots x_{k})e_{k+1} = e_{1}x_{1}\cdots x_{k}$,
  since an idempotent is a right identity for its
  $\L$-class~\cite[Proposition~2.3.3]{Howie1995aa}.
  Therefore $x =\prod_{k = 1}^{n} e_{k} x_{k}$.  Furthermore, for each $k \in
  \{1,\ldots,n\}$ the element $e_{k}x_{k}$ is contained in $T$ by assumption,
  and is contained in $J$ since $$J = J_{e_{k}} \geq J_{e_{k}x_{k}} \geq J_{x} =
  J.$$ By repeated application of condition~\eqref{item-assumption-2}, it
  follows that $x = \prod_{k = 1}^{n} e_{k} x_{k} \in T \cap J$, and so
  $E\genset{X^{\prime}} \subseteq T$.

  Note that condition~\eqref{item-assumption-2} is equivalent to the statement
  that $(T \cap J) \cup \{0\}$ is a subsemigroup of $J^{*}$.  Since $T$
  intersects every $\H$-class of $J$ and $J$ is finite, it follows that $T$
  contains every idempotent of $J$.
  
  To prove that condition~\eqref{item-assumption-1} of
  Lemma~\ref{lem-subsemigroup} holds, let $x \in \genset{X^{\prime}} \cap J$.
  Since $J$ is a regular $\J$-class, there exists an idempotent $f \in T\cap J$
  such that $fx = x$.  By definition of the set $E$, there exists an idempotent
  $e \in E$ such that $e \L f$, and so $x = fx = (fe)x = f(ex)$.  We have $ex
  \in T\cap J$ since $E\genset{X^{\prime}} \subseteq T$, and since $(T
  \cap J)\cup \{0\}$ is a subsemigroup of $J^{*}$, it follows that $x = f(ex)
  \in T$.

  To prove that condition~\eqref{item-assumption-3} of
  Lemma~\ref{lem-subsemigroup} holds, let $x \in T \cap J$ and $y \in \genset{X
  ^ {\prime}}$.  First suppose that $xy \in J$.  By assumption, there exists an
  idempotent $e \in E$ such that $x = xe$. Since $xy = x(ey) \in J$ it follows
  that $$J = J_{e} \geq J_{ey} \geq J_{x(ey)} = J_{xy} = J,$$ and so $ey \in J$.
  Furthermore, $ey \in E \genset{X ^ {\prime}} \subseteq T$. Since $x, ey \in T
  \cap J$, $xy \in J$, and $(T \cap J)\cup \{0\}$ is a subsemigroup of $J^{*}$,
  it follows that $xy = x(ey) \in T \cap J$.  Finally suppose that $yx \in J$.
  Since $J$ is a regular $\J$-class, there exists an idempotent $f \in T \cap J$
  such that $f(yx) = yx$. By definition of $E$, there exists $e \in E$ such that
  $fe = f$, and so $yx = f(yx) = (fe)yx = f(ey)x$. Note that $E\genset{X ^
  {\prime}} \subseteq T$ implies that $ey \in T$, and $ey \in J$ since 
  $$J = J_{e} \geq J_{ey} \geq J_{feyx} = J_{yx} = J.$$ 
  Finally, $f, ey, x \in T \cap
  J$ and $(T \cap J) \cup \{0\}$ is a subsemigroup of $J^{*}$, and it follows
  that $yx \in T$.
\end{proof}

Note that dual results hold if we replace $E$ by a set consisting of one
idempotent from each $\R$-class of $J$ and we replace $EX^{\prime}$ by
$X^{\prime}E$.

Let $\phi$ be the isomorphism from  $J ^ *$ to a Rees 0-matrix semigroup
defined at the start of Section~\ref{section-arbitrary} and  suppose that $T$
is a subset of $S$ such that $S \setminus T \subseteq J$ and $T$ intersects
every $\H$-class of $J$ non-trivially.  Condition~\eqref{item-assumption-2} of
Lemma~\ref{lem-subsemigroup} is equivalent to the statement that $(T \cap
J)\cup \{0\}$ is a subsemigroup of $J^{*}$, the principal factor of $J$. Since
$\phi$ is an isomorphism, this is equivalent to the statement that $(T \cap
J)\phi \cup \{0\}$ is a subsemigroup of $(J^{*})\phi$. 

The following corollary is the main result of this section.

\begin{cor}\label{cor-intersect}
  Let $S$ be a finite semigroup, and let $T$ be a subset of $S$ such that $S
  \setminus T$ is contained in a regular $\J$-class $J$ of $S$. Suppose that $T$
  intersects every $\H$-class of $S$ and let $E$ be a set consisting of one
  idempotent from each $\L$-class of $J$.  Then $T$ is a maximal subsemigroup of
  $S$ if and only if $(T \cap J)\phi \cup \{0\}$ is a maximal subsemigroup of
  $(J^{*})\phi$ containing $(EX')\phi$.
\end{cor}

\begin{proof}
  $(\Rightarrow)$
  Let $U$ be a subset of $J$ such that $U \cup \{0\}$ is a
  maximal subsemigroup of $(J^{*})\phi$ containing $(T \cap J)\phi \cup
  \{0\}$.  Then by Lemma~\ref{lem-intersect}, the set $M = (S \setminus J) \cup
  U$ is a proper subsemigroup of $S$ containing $T$. Since $T$ is maximal, it
  follows that $T = M$, and $U = T \cap J$.

  $(\Leftarrow)$
  Let $M$ be a maximal subsemigroup of $S$ containing $T$.
  Then $(M \cap J)\phi \cup \{0\}$ is a proper subsemigroup of $(J^{*})\phi$
  containing $(T \cap J)\phi \cup \{0\}$. Since the latter is a maximal
  subsemigroup of $(J^{*})\phi$, it follows that $T \cap J = M \cap J$, and
  hence $T = M$.
\end{proof}

To calculate the maximal subsemigroups of type~\eqref{item-every} arising from
$J$, we compute the set $EX^{\prime}$, we compute an isomorphism $\phi$ from
$J^{*}$ to a normalized regular Rees 0-matrix semigroup $R$, and then we search
for the maximal subsemigroups of $R$ containing $(EX^{\prime})\phi$ that
intersect every $\H$-class of $R$.  Computing the maximal subsemigroups of $R$
that have non-trivial intersection with every $\H$-class was the subject of
Algorithm~\ref{algorithm-type-6}. One approach would be to simply compute all of
the maximal subsemigroups using Algorithm~\ref{algorithm-type-6} and then
discard those which do not contain $(EX^{\prime})\phi$. It is possible to modify
Algorithm~\ref{algorithm-type-6} to compute directly only those maximal
subsemigroups containing $(EX^{\prime})\phi$. However, we will not do this in
detail.

A generating set for any such maximal subsemigroup is given by a generating set
for $S \setminus J$, along with the preimage under $\phi$ of a generating set
for the maximal subsemigroup of $R$ (minus the element $0 \in R$).
Algorithm~\ref{algorithm-type-6} produces generating sets for the maximal
subsemigroups that it finds. 


\subsection{Maximal subsemigroups arising from regular $\J$-classes that
are unions of $\H$-classes: types~\eqref{item-both}--\eqref{item-empty}}

In this section, we consider those maximal subsemigroups of a finite semigroup
$S$ that arise from the exclusion of elements in a regular $\J$-class of $S$,
and that are unions of $\H$-classes of $S$.  The principal purpose of this
section is to give necessary and sufficient conditions for a subset of a finite
semigroup to be a maximal subsemigroup in terms of the properties of certain
associated graphs. The formulation in terms of graphs makes the problem of
computing maximal subsemigroups of this type more tractable. In particular, we
can take advantage of several well-known algorithms from graph theory, such as
those for computing strongly connected components (see~\cite{Gabow2000aa,
Tarjan1972aa} or~\cite[Section 4.2]{Sedgewick2011aa}) and finding all maximal
cliques~\cite{Bron1973}.

A \textit{digraph} is a pair $(V, E)$ consisting of a set of \textit{vertices}
$V$ and a set of \textit{edges} $E \subseteq V \times V$ such that $(u,v) \in E$
implies $u \not =v$.  An edge $(u, v) \in E$ is \textit{an edge from $u$ to
$v$}.  
If $\Gamma = (V, E)$ is a digraph and $W$ is a subset of the vertices of
$\Gamma$, then the \textit{subdigraph induced by $W$} is the digraph with
vertices $W$ and edges $\set{(u, v)\in E}{u, v\in W}$.
A vertex $v$ in a digraph $\Gamma = (V, E)$ is a \textit{source} if $(u,
v)\not\in E$ for all $u\in V$, and is a \textit{sink} if $(v, u)\not\in E$ for
all $u\in V$.  If $\Gamma = (V, E)$ is a digraph, then a \textit{path} is a
sequence of distinct vertices $(v_{1}, \ldots, v_{m})$, $m \geq 1$, of $\Gamma$
such that $(v_{i}, v_{i+1})\in E$ for all $i \in \{1, \ldots, m - 1\}$.   A
path $(v_{1}, \ldots, v_{m})$ is said to be \textit{a path from $v_{1}$ to
$v_{m}$ in $\Gamma$}.  A vertex $v$ is \textit{reachable} from a vertex $u$ if
$u = v$, or there is a path from $u$ to $v$ in $\Gamma$. The \textit{strongly
connected component} of a vertex $v$ of a digraph is the set consisting of all
vertices $u$ such that there is a path from $u$ to $v$ and a path from $v$ to
$u$. Note that our definition of digraphs does not allow loops.  A
\textit{colouring} of a digraph $\Gamma = (V, E)$ is just a function $c: V\to
\N$.

Throughout this section we suppose that $J$ is a regular $\J$-class of the
finite semigroup $S = \genset{X}$. Recall that $X'$ consists of those
generators $x\in X$ such that $J_x > J$.

The following digraphs are central to the results in this section. We define
$\Gamma_{\L}$ to be the quotient of the digraph with vertices $J/\L$ and edges
$$\bigset{ (L_{a}, L_{b})\in J/\L\times J/\L }{ L_{a}x = L_{b},\ L_a\not =L_b,\
\text{for some}\ x \in X ^ {\prime} }$$ 
by its strongly connected components. We define a colouring $c$
of $\Gamma_{\L}$ so that any vertex $V$ containing an $\L$-class which has
non-empty intersection with $\genset{X'}$ has $c(V) = 1$ and every other vertex
$U$ has $c(U) = 0$. The digraph $\Gamma_{\R}$ is defined dually.  Note that
$\Gamma_{\L}$ and $\Gamma_{\R}$ are acyclic digraphs.

We require two additional graphs.  The first graph, $\Delta$ is isomorphic to a
quotient of the Graham-Houghton graph of the principal factor of $J$. We define
$\Delta$ to have vertex set equal to  the disjoint union of the vertices of
$\Gamma_{\L}$ and $\Gamma_{\R}$, and edges $\{U, V\}$ if the intersection of
some $\L$-class in $U$ with an $\R$-class in $V$, or vice versa, is a group
$\H$-class.  The second graph, $\Theta$, has the same vertex set as $\Delta$,
and it has an edge incident to $U$ and $V$ if there is an element of
$\genset{X'}$ in the intersection of some $\L$-class in $U$ with an $\R$-class
in $V$, or vice versa. Both of the graphs $\Delta$ and $\Theta$ are bipartite.

\begin{ex}\label{ex-sec-3}
  The \textit{full transformation monoid} $\mathcal{T}_n$, for $n\in \N$, is
  the semigroup consisting of all mappings from $\{1, \ldots, n\}$ to itself
  (called \textit{transformations}) with the operation of composition of
  functions.  This semigroup plays an analogous role in semigroup theory as
  that played by the symmetric group in group theory. 
  
  Let $W$ be the subsemigroup of $\mathcal{T}_7$ generated by the
  transformations
  \begin{align*}
    x_{1} = \begin{pmatrix}
      1 & 2 & 3 & 4 & 5 & 6 & 7 \\
      1 & 3 & 4 & 1 & 5 & 5 & 5
    \end{pmatrix},\quad &
    x_{2} = \begin{pmatrix}
      1 & 2 & 3 & 4 & 5 & 6 & 7 \\
      1 & 4 & 1 & 3 & 5 & 5 & 5
    \end{pmatrix},\\
    x_{3} = \begin{pmatrix}
      1 & 2 & 3 & 4 & 5 & 6 & 7 \\
      3 & 3 & 1 & 2 & 5 & 5 & 5
    \end{pmatrix},\quad &
    x_{4} = \begin{pmatrix}
      1 & 2 & 3 & 4 & 5 & 6 & 7 \\
      4 & 4 & 2 & 3 & 5 & 5 & 5
    \end{pmatrix},\\
    x_{5} = \begin{pmatrix}
      1 & 2 & 3 & 4 & 5 & 6 & 7 \\
      1 & 1 & 3 & 4 & 5 & 5 & 6
    \end{pmatrix},\quad &
    x_{6} = \begin{pmatrix}
      1 & 2 & 3 & 4 & 5 & 6 & 7 \\
      1 & 2 & 2 & 4 & 5 & 6 & 7
    \end{pmatrix},\\
    x_{7} = \begin{pmatrix}
      1 & 2 & 3 & 4 & 5 & 6 & 7 \\
      1 & 4 & 3 & 4 & 5 & 6 & 7
    \end{pmatrix},\quad &
    x_{8} = \begin{pmatrix}
      1 & 2 & 3 & 4 & 5 & 6 & 7 \\
      1 & 2 & 4 & 4 & 5 & 6 & 7
    \end{pmatrix}.
  \end{align*}

  Let $J$ be the $\J$-class $J_{x_{1}}$.  The following calculations were
  performed with the {\sc GAP}~\cite{GAP4} package {\sc
  Semigroups}~\cite{Mitchell2017aa}. The $\J$-class $J$ is regular, and contains
  the generators $x_{1}, x_{2}, x_{3}$, and $x_{4}$. The remaining generators
  are contained in $\J$-classes that are above $J$ in the $\J$-class partial
  order, and so $X^{\prime} = \{x_{5},x_{6},x_{7},x_{8}\}$. The set of
  $\L$-classes of $J$ is $J / \L = \{ L_{x_{1}}, L_{x_{3}}, L_{x_{4}},
  L_{x_{1}x_{6}} \}$, and the set of $\R$-classes of $J$ is $J / \R = \{
  R_{x_{1}}, R_{x_{2}}, R_{x_{3}}, R_{x_{8}x_{2}}, R_{x_{6}x_{2}},
  R_{x_{7}x_{3}} \}$.

  There are $4$ strongly connected components of $\L$-classes in $J$ each
  consisting of a single $\L$-class. Hence 
  the digraph $\Gamma_{\L}$ has $4$ vertices, one for each strongly connected
  component.  There are also $4$ strongly connected components of $\R$-classes:
  two of which have size $1$, while the other two each have size $2$.
  The digraphs $\Gamma_{\L}$ and $\Gamma_{\R}$ are depicted in
  Figure~\ref{fig-gamma}.

  \begin{figure}
    \begin{center}
    \begin{tikzpicture}
      \node (1) at (5, 4)  {$\{L_{x_{4}}\}$};
      \node (2) at (0, 4)  {$\{L_{x_{3}}\}$};
      \node (3) at (3, 2)  {$\{L_{x_{1}}\}$};
      \node (4) at (3, 0)  {$\{L_{x_{1}x_{6}}\}$};

      \arc{1}{3};
      \arc{2}{3};
      \arc{2}{4};
      \arc{3}{4};

      \node (11) at (8,    2)  {$\{R_{x_{1}}\}$};
      \node (12) at (10.5, 4)  {$\{R_{x_{2}}\}$};
      \node (13) at (10.5, 2)  {$\{R_{x_{8}x_{2}}, R_{x_{6}x_{2}}\}$};
      \node (14) at (10.5, 0)  {$\{R_{x_{3}}, R_{x_{7}x_{3}}\}$};

      \arc{12}{13};
      \arc{13}{14};
    \end{tikzpicture}
    \end{center}
    \caption{The digraphs $\Gamma_{\L}$ and $\Gamma_{\R}$ from
    Example~\ref{ex-sec-3}.}\label{fig-gamma}
  \end{figure}
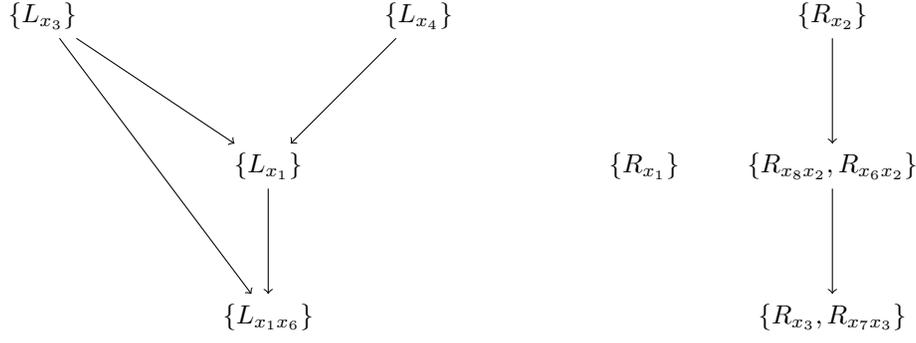

  Since $\Gamma_{\L}$ and $\Gamma_{\R}$ each have four vertices, it follows that
  the bipartite graphs $\Delta$ and $\Theta$ each have eight vertices. These
  graphs are shown in Figures~\ref{fig-delta} and~\ref{fig-theta}.  The set of
  edges of $\Delta$ was determined by computation of the idempotents in $J$.
  There are four elements in $\genset{X^{\prime}} \cap J$: $x_{5}^{2} \in
  L_{x_{1}} \cap R_{x_{3}}$, $x_{4}x_{1}x_{6} \in L_{x_{1}x_{6}} \cap
  R_{x_{3}}$, $x_{7}x_{4}x_{1}x_{6} \in L_{x_{1}x_{6}} \cap R_{x_{7}x_{3}}$, and
  $x_{7}x_{4}x_{1} \in L_{x_{1}} \cap R_{x_{7}x_{3}}$. The $\L$- and
  $\R$-classes of these elements determine the edges in the graph $\Theta$,
  along with the colours of the vertices in $\Gamma_{\L}$ and $\Gamma_{\R}$. The
  vertices of $\Gamma_{\L}$ with colour $1$ are $\{L_{x_{1}}\}$ and
  $\{L_{x_{1}x_{6}}\}$, whilst the only vertex of $\Gamma_{\R}$ with colour $1$
  is $\{R_{x_{3}}, R_{x_{7}x_{3}} \}$.

  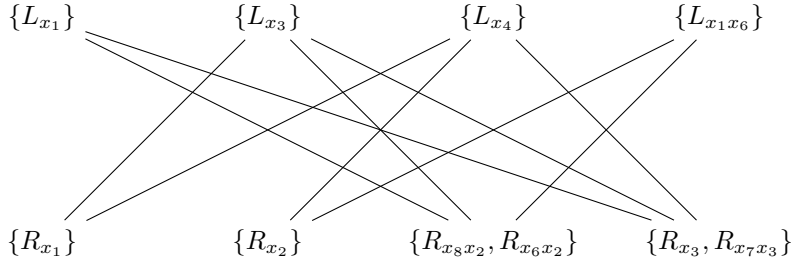
\begin{figure}
    \begin{center}
    \begin{tikzpicture}
      \node (1) at (0,  3)  {$\{L_{x_{1}}\}$};
      \node (2) at (3,  3)  {$\{L_{x_{3}}\}$};
      \node (3) at (6,  3)  {$\{L_{x_{4}}\}$};
      \node (4) at (9,  3)  {$\{L_{x_{1}x_{6}}\}$};

      \node (a) at (0,  0)  {$\{R_{x_{1}}\}$};
      \node (b) at (3,  0)  {$\{R_{x_{2}}\}$};
      \node (c) at (6,  0)  {$\{R_{x_{8}x_{2}}, R_{x_{6}x_{2}}\}$};
      \node (d) at (9,  0)  {$\{R_{x_{3}}, R_{x_{7}x_{3}}\}$};

      \edge{1}{c};
      \edge{1}{d};
      \edge{2}{a};
      \edge{2}{c};
      \edge{2}{d};
      \edge{3}{a};
      \edge{3}{b};
      \edge{3}{d};
      \edge{4}{b};
      \edge{4}{c};
    \end{tikzpicture}
    \end{center}
    \caption{The graph $\Delta$ from Example~\ref{ex-sec-3}.}\label{fig-delta}
  \end{figure}

  \begin{figure}
    \begin{center}
    \begin{tikzpicture}
      \node (1) at (0,  3)  {$\{L_{x_{1}}\}$};
      \node (2) at (3,  3)  {$\{L_{x_{3}}\}$};
      \node (3) at (6,  3)  {$\{L_{x_{4}}\}$};
      \node (4) at (9,  3)  {$\{L_{x_{1}x_{6}}\}$};

      \node (a) at (0,  0)  {$\{R_{x_{1}}\}$};
      \node (b) at (3,  0)  {$\{R_{x_{2}}\}$};
      \node (c) at (6,  0)  {$\{R_{x_{8}x_{2}}, R_{x_{6}x_{2}}\}$};
      \node (d) at (9,  0)  {$\{R_{x_{3}}, R_{x_{7}x_{3}}\}$};

      \edge{1}{d};
      \edge{4}{d};
    \end{tikzpicture}
    \end{center}
    \caption{The graph $\Theta$ from Example~\ref{ex-sec-3}.}\label{fig-theta}
  \end{figure}
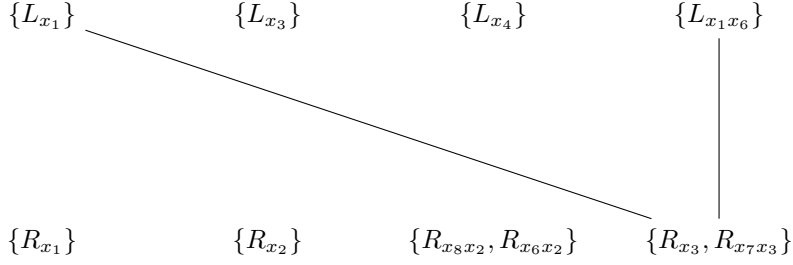
\end{ex}

The following lemma is a straightforward consequence of Green's Lemma and
Lemma~\ref{lem-d-finite}.

\begin{lem}\label{lem-path-gamma}
  \begin{enumerate}[label=\emph{(\roman*)},ref=\roman*]
    \item 
      The vertex containing $L_{b} \in J/\L$ in $\Gamma_{\L}$ is reachable from
      the vertex containing $L_{a}$ if and only if $L_{a} = L_{b}$, or there
      exists $s \in L_{a}$ and $x \in \genset{S \setminus J}$ such that $sx \in
      L_{b}$.

    \item 
      The vertex containing $R_{b} \in J/\R$ in $\Gamma_{\R}$ is reachable
      from the vertex containing $R_{a}$ if and only if $R_{a} = R_{b}$, or there
      exists $s \in R_{a}$ and $x \in \genset{S \setminus J}$ such that $xs \in
      R_{b}$.

  \end{enumerate}
\end{lem}

By Green's Lemma [Lemma~\ref{lem-green}] and Lemma~\ref{lem-d-finite}, if $T$ is
a subsemigroup of $S$ such that $S\setminus T\subseteq J$, then $T \cap J$
contains an $\L$-class $L$ if and only if $T\cap J$ contains every $\L$-class in
every vertex of $\Gamma_{\L}$ that is reachable from the vertex of $\Gamma_{\L}$
containing $L$.  The analogous statement holds for $\R$-classes.  Furthermore,
$T\cap J$ contains an $\H$-class $H$ if and only if $T\cap J$ contains every
$\H$-class that is the intersection of an $\L$- and an $\R$-class that are
contained in vertices that are reachable from the vertices containing the $\L$-
and $\R$-classes containing $H$ in $\Gamma_{\L}$ and $\Gamma_{\R}$,
respectively.

\begin{ex}
  Let $W$ and $J$ be the semigroup and the $\J$-class, respectively, from
  Example~\ref{ex-sec-3}. Suppose that $T$ is a subsemigroup of $W$ such that $W
  \setminus T \subseteq J$.  By analysing the digraph $\Gamma_{\L}$, depicted in
  Figure~\ref{fig-gamma}, we see that if $T$ contains the $\L$-class
  $L_{x_{4}}$, then $T$ also contains the $\L$-classes $L_{x_{1}}$ and
  $L_{x_{1}x_{6}}$, since the vertices containing these $\L$-classes are
  reachable in $\Gamma_{\L}$ from the vertex containing $L_{x_{4}}$.  Likewise,
  if $T$ contains the $\R$-class $R_{x_{8}x_{2}}$, then by considering the
  digraph $\Gamma_{\R}$, depicted in Figure~\ref{fig-gamma}, we see that $T$
  also contains the $\R$-classes $R_{x_{6}x_{2}}, R_{x_{3}}$, and
  $R_{x_{7}x_{3}}$.  If we consider these digraphs together, then we see that
  $T$ contains the $\H$-class $L_{x_{1}} \cap R_{x_{3}}$ if and only if $T$ also
  contains the $\H$-classes $L_{x_{1}} \cap R_{x_{7}x_{3}}$, $L_{x_{1}x_{6}}
  \cap R_{x_{3}}$, and $L_{x_{1}x_{6}} \cap R_{x_{7}x_{3}}$. 
\end{ex}

The digraphs $\Gamma_{\L}$, $\Gamma_{\R}$, and the bipartite graphs $\Delta$
and $\Theta$ can be created using graph algorithms applied to the left and
right Cayley graphs of the semigroup $S$ with respect to its generating set
$X$.  The time complexity of finding $\Gamma_{\L}$, $\Gamma_{\R}$, $\Delta$,
and $\Theta$ using the left and right Cayley graphs of $S$ is $O(|S||X|)$.
This is the same as the time complexity of determining the left and right
Cayley graphs of $S$ using, say, the Froidure-Pin
Algorithm~\cite{Froidure1997aa}. However in practice finding $\Gamma_{\L}$,
$\Gamma_{\R}$, $\Delta$, and $\Theta$ using the Cayley graphs of $S$ will be
much faster than determining the Cayley graphs themselves.  For certain types
of semigroups, such as semigroups represented by a generating set consisting of
transformations of a finite set, the $\J$-class itself and the vertices and
edges of $\Gamma_{\L}$, $\Gamma_{\R}$, and $\Delta$ can be determined without
finding either the left or right Cayley graph of $S$; see~\cite{East2015ab} for
further details.

In what follows, we will assume that $\Gamma_{\L}$, $\Gamma_{\R}$, $\Delta$,
and $\Theta$ and the colourings of $\Gamma_{\L}$ and $\Gamma_{\R}$ are known
\textit{a priori}. 


Suppose that $T$ is a subset of $S$ such that $S \setminus T \subseteq J$ and
suppose that there exist proper subsets $A \subsetneq J/\L$ and $B \subsetneq
J/\R$ such that $T\cap J$ is the union of the Green's classes in $A$ and in
$B$.  If $T$ is a maximal subsemigroup of $S$, then type~\eqref{item-both} is
when  $A \not= \varnothing$ and $B \not = \varnothing$; type~\eqref{item-l}
corresponds to $A \not= \varnothing$ and $B = \varnothing$; type~\eqref{item-r}
is when $A = \varnothing$ and $B \not = \varnothing$; and
type~\eqref{item-empty} is when $A = \varnothing$ and $B = \varnothing$.

Recall that the vertices of $\Gamma_{\L}$ and $\Gamma_{\R}$ are
sets of $\L$- and $\R$-classes of $J$, respectively.

\begin{prop}\label{prop-rect}
  Let $S$ be a finite semigroup, and let $T$ be a proper subset of $S$ such that
  $S\setminus T$ is contained in a regular $\J$-class $J$ of $S$.  Suppose that
  there exist proper subsets $A \subsetneq J/\L$ and $B \subsetneq J/\R$ such
  that $T\cap J$ is the union of the Green's classes in $A$ and in $B$. Then $T$
  is a subsemigroup of $S$ if and only if the following hold:

  \begin{enumerate}[label=\emph{(\roman*)},ref=\roman*]

    \item\label{item-prop-rect-1}
      $A$ is a union of vertices of $\Gamma_{\L}$, and 
      $B$ is a union of vertices of $\Gamma_{\R}$;

    \item\label{item-prop-rect-2}
      if $U$ and $V$ are vertices of $\Gamma_{\L}$ such that 
      $U$ is contained in $A$ and there is an edge from $U$ to $V$ in
      $\Gamma_{\L}$, then $V$ is contained in $A$;
    
    \item\label{item-prop-rect-3}
      if $U$ and $V$ are vertices of $\Gamma_{\R}$ such that 
      $U$ is contained in $B$ and there is an edge from $U$ to $V$ in
      $\Gamma_{\R}$, then $V$ is contained in $B$;

    \item\label{item-prop-rect-4}
      if $\{U, V\}$ is an edge in $\Theta$, then either $U$ or $V$, or both,
      is contained in $A\cup B$;
    
    \item\label{item-prop-rect-5}
      the vertices contained in $A\cup B$ form an independent set of $\Delta$.

  \end{enumerate}
\end{prop}

\begin{proof}
  ($\Rightarrow$)
  As mentioned after Lemma~\ref{lem-path-gamma},  if $T\cap J$ contains an
  $\L$-class $L$, then $T \cap J$ contains every $\L$-class in every vertex of
  $\Gamma_{\L}$ that is reachable from the vertex containing $L$. An analogous
  statement holds for $\R$-classes and $\Gamma_{\R}$.
  Parts~\eqref{item-prop-rect-1},~\eqref{item-prop-rect-2},
  and~\eqref{item-prop-rect-3} follow immediately from these observations.

  If $\{U, V\}$ is an edge in $\Theta$, then by definition, there is an
  element  $x\in \genset{X'}$ in the intersection of some $\L$-class $L$ in $U$
  and some $\R$-class $R$ in $V$.  By Lemma~\ref{lem-subsemigroup}, since $T$
  is a subsemigroup, $x\in T$ and so either $L\in A$ and $U$ is contained in
  $A$; or $R\in B$ and $V$ is contained in $B$. Therefore part (iv) holds. 
 
  If $A = \varnothing$ or $B = \varnothing$, then part (v) holds immediately by
  the definition of $\Delta$, so suppose otherwise.  Since $T$ is a
  proper subsemigroup of $S$, it follows that it is contained in a maximal
  subsemigroup $M$ of $S$.   By the assumption that $T\cap J$ is a union of
  non-empty sets of $\L$- and $\R$-classes, $M$ must be of type described
  in Proposition~\ref{prop-graham}(c)\eqref{item-prop-rect}.  Hence $(M\cap
  J)\phi = (I \times G \times \Lambda) \setminus (I' \times G \times \Lambda')$
  for some non-empty sets $I'\subsetneq I$ and $\Lambda'\subsetneq \Lambda$,
  and $(M\cap J)\phi\cup \{0\}$ is a maximal subsemigroup of $(J)\phi \cup
  \{0\}$ of type~\eqref{item-rect}. We may assume without loss of generality
  that $I = J/\R$ and $\Lambda = J/\L$, and that for an $\L$-class $L$ of $J$
  and $\R$-class $R$ of $J$, $(L)\phi = I \times G \times \{L\}$ and $(R)\phi =
  \{R\} \times G \times \Lambda$.

  Proposition~\ref{prop-rms} implies that there exist
  non-empty sets $X \subsetneq I$ and $Y\subsetneq \Lambda$ such that $I' =
  I\setminus X$, $\Lambda' = \Lambda\setminus Y$, and $X\cup Y$ is a maximal
  independent set in the Graham-Houghton graph of  $(J)\phi \cup \{0\}$.
  Since $M$ contains the union of the $\L$-classes in $A$, it follows that $A
  \subseteq Y$, and similarly, $B\subseteq X$. 
  
  If $U$ and $V$ are vertices of $\Delta$ 
  contained in $A\subseteq Y$ and $B\subseteq X$, respectively, then, since
  $\Delta$ is a quotient of the Graham-Houghton graph of $(J)\phi\cup \{0\}$,
  there is no edge in $\Delta$ incident to $U$ and $V$. Thus (v) holds. 
  
  ($\Leftarrow$)
  It suffices to show that $T$ satisfies the
  conditions~\eqref{item-assumption-1},~\eqref{item-assumption-2},
  and~\eqref{item-assumption-3} of Lemma~\ref{lem-subsemigroup}.
 
  To verify that condition~\eqref{item-assumption-1} of
  Lemma~\ref{lem-subsemigroup} holds, let $x \in \genset{X'}\cap J$. Then there
  exists an edge in $\Theta$  between the vertex $U$ containing $L_x$ and the
  vertex $V$ containing $R_x$ (by the definition of $\Theta$). By
  assumption~\eqref{item-prop-rect-4} of this proposition, either $U \subseteq
  A$ or $V \subseteq B$ (or both).  In the first case, $x \in L_x \in A$ and $A
  \subseteq T$ by assumption. The other case is similar.
  
  For the second condition, suppose that $x, y \in T \cap J$ and $xy \in J$.  By
  Lemma~\ref{lem-d-finite}, $xy \in R_{x} \cap L_{y}$. It follows
  by~\cite[Proposition~2.3.7]{Howie1995aa} that the $\H$-class $L_{x} \cap
  R_{y}$ contains an idempotent, and is therefore a group.  By
  assumption~\eqref{item-prop-rect-5}, the vertices contained in $A \cup B$
  form an independent set of $\Delta$. Hence either $L_{x} \notin A$ or $R_{y}
  \notin B$.  If $R_{y} \notin B$, then, since $T \cap J$ is a union of $\L$-
  and $\R$-classes of $J$ and $y\in T\cap J$, $L_{y} \in A$.  By
  Lemma~\ref{lem-d-finite}, $L_{xy} = L_{y} \in A$, and $xy \in T$.  If $L_{x}
  \notin A$, then the proof is analogous.

  To show that the final condition of Lemma~\ref{lem-subsemigroup} holds, let
  $x \in T \cap J$ and $y \in \genset{X ^ {\prime}}$. Note that since $x \in T
  \cap J$, either $L_{x} \in A$ or $R_{x} \in B$. Suppose that $xy \in J$. If
  $L_{x} \in A$, then since the vertex containing $L_{xy} = L_{x}y$ is reachable
  in $\Gamma_{\L}$ from the vertex containing $L_{x}$, it follows that $L_{xy}
  \in A$ and so $xy\in T$. Otherwise $R_{x} \in B$, and so $R_{xy} = R_{x} \in
  B$ and $xy \in T$.  The proof that $yx \in J$ implies $yx \in T$ is similar.
\end{proof}


\subsubsection{Maximal subsemigroups from maximal rectangles:
type~\eqref{item-both}}\label{subsec-rect}

The following corollary, a slight adaptation of Proposition~\ref{prop-rect},
gives necessary and sufficient conditions for a maximal subsemigroup of
type~\eqref{item-both} to exist.

\begin{cor}\label{cor-rect}
  Suppose that $A \neq \varnothing$ and $B \neq \varnothing$.  Then $T$ is a
  maximal subsemigroup of $S$ if and only if
  conditions~\eqref{item-prop-rect-1}--\eqref{item-prop-rect-4} of
  Proposition~\ref{prop-rect} hold, and the vertices contained in $A \cup B$
  form a maximal independent set of $\Delta$.
\end{cor}

\begin{proof}
  ($\Rightarrow$)
  Since $T$ is itself the only maximal subsemigroup of $S$ containing $T$, then,
  as described in the proof of Proposition~\ref{prop-rect} (forward
  implication), the set $A\cup B$ is a maximal independent set in
  the Graham-Houghton graph of $(J)\phi \cup \{0\}$.  Since $\Delta$ is a
  quotient of the Graham-Houghton graph, this implies the maximality of the
  independent set in $\Delta$.

  ($\Leftarrow$)
  Let $M$ be a maximal subsemigroup of $S$ containing $T$.  By the assumption
  that $T\cap J$ is a union of non-empty sets of $\L$- and $\R$-classes, $M$
  must be of the type described in
  Proposition~\ref{prop-graham}(c)\eqref{item-prop-rect}. By
  Proposition~\ref{prop-rect}, $M$ corresponds to an independent set of
  $\Delta$ that contains the vertices contained in $A \cup B$.  Since this
  latter is a maximal independent set in $\Delta$, these sets are equal.  Since
  the subsemigroups $T$ and $M$ are determined by their corresponding
  independent sets, it follows that $T = M$.
\end{proof}

We describe, in the following steps, how to use Corollary~\ref{cor-rect} to
compute the maximal subsemigroups of type~\eqref{item-both} corresponding to a
given regular $\J$-class. 

The first step is to determine the maximal independent sets in the bipartite
graph $\Delta$, or equivalently to find the maximal cliques in the complement
of $\Delta$.  The Bron-Kerbosch Algorithm~\cite{Bron1973} (implemented in~the
{\sc GAP}~\cite{GAP4} package {\sc Digraphs}~\cite{DeBeule2017aa}) is a recursive
algorithm for finding maximal cliques in a graph. Roughly speaking, this
algorithm proceeds by attempting to extend a given clique by another vertex. By
Proposition~\ref{prop-rect}\eqref{item-prop-rect-2}
and~\eqref{item-prop-rect-3}, in the search for maximal cliques in the
complement of $\Delta$, we are only interested in those cliques $K$ with
following property: if $U$ is a vertex in $K$ and $V$ is a vertex of $\Delta$,
such that there is an edge in $\Gamma_{\R}$ or $\Gamma_{\L}$ from $U$ to $V$,
then $V$ is in $K$ also.  As such, the search tree in the Bron-Kerbosch
Algorithm can be pruned to exclude any branch starting at a clique containing a
vertex $U$ for which there is a vertex $V$ in $\Gamma_{\R}$, or $\Gamma_{\L}$,
which is reachable from $U$ but does not extend the given clique, or where we
have already discovered every maximal clique containing $V$. 

We thereby produce a collection of maximal independents sets in $\Delta$, each
of which gives rise to sets of vertices $A$ and $B$ that satisfy
Proposition~\ref{prop-rect}\eqref{item-prop-rect-1},~\eqref{item-prop-rect-2},
and~\eqref{item-prop-rect-3}.
The second step is then to check which of these sets $A$ and $B$ satisfy
part~\eqref{item-prop-rect-4} of Proposition~\ref{prop-rect}, which is routine.
Given sets $A$ and $B$ satisfying all the conditions in
Proposition~\ref{prop-rect}, the final step is to specify a generating set for
the maximal subsemigroup $T$; see Proposition~\ref{prop-genset-rect}.

\begin{ex}\label{ex-rect}
  Let $W$ and $J$ be the semigroup and the $\J$-class, respectively, from
  Example~\ref{ex-sec-3}. Consider the graph $\Delta$, which is shown in
  Figure~\ref{fig-delta}. 
  
  There are $7$ maximal independent subsets of $\Delta$ in total: two
  correspond to the vertices of $\Gamma_{\L}$ and $\Gamma_{\R}$; three
  further correspond to sets $A$ and $B$ that do not satisfy either of
  Proposition~\ref{prop-rect}\eqref{item-prop-rect-2}
  or~\eqref{item-prop-rect-3}; the remaining two maximal independent subsets
  of $\Delta$ correspond to non-empty sets $A$ and $B$ satisfying
  Proposition~\ref{prop-rect}\eqref{item-prop-rect-1},~\eqref{item-prop-rect-2},
  and~\eqref{item-prop-rect-3}.

  The first such subset is $\{ \{ L_{x_{1}} \}, \{
  L_{x_{1}x_{6}} \}, \{ R_{x_{1}} \} \}$, which corresponds to the sets $A_{1} =
  \{ L_{x_{1}}, L_{x_{1}x_{6}} \} \subseteq J / \L$ and $B_{1} = \{ R_{x_{1}} \}
  \subseteq J / \R$; the second such subset corresponds to the sets $A_{2} = \{
  L_{x_{1}x_{6}} \}$ and $B_{2} = \{ R_{x_{1}}, R_{x_{3}}, R_{x_{7}x_{3}} \}$.

  There are two edges in the graph $\Theta$, as shown in Figure~\ref{fig-theta}:
  $\{ \{ L_{x_{1}} \}, \{ R_{x_{3}}, R_{x_{7}x_{3}} \} \}$, and $\{ \{
  L_{x_{1}x_{6}} \}, \{ R_{x_{3}}, R_{x_{7}x_{3}} \} \}$.  For the first edge,
  $\{L_{x_{1}}\} \subseteq A_{1}$ and $\{R_{x_{3}}, R_{x_{7}x_{3}}\} \subseteq
  B_{2}$; for the second edge, $\{L_{x_{1}x_{6}} \} \subseteq A_{1}$ and
  $\{R_{x_{3}}, R_{x_{7}x_{3}}\} \subseteq B_{2}$.  In other words, the sets
  $A_{1}$ and $B_{1}$, and the sets $A_{2}$ and $B_{2}$, satisfy
  Proposition~\ref{prop-rect}\eqref{item-prop-rect-4}.

  Therefore there are two maximal subsemigroups of $W$ of type~\eqref{item-both}
  arising from $J$: the set consisting of $W \setminus J$ along with the union
  of the $\L$-classes in $A_{1}$ and the union of the $\R$-classes in $B_{1}$,
  and the set consisting of $W \setminus J$ along with the union of the
  $\L$-classes in $A_{2}$ and the union of the $\R$-classes in $B_{2}$.
\end{ex}

\begin{prop}\label{prop-genset-rect}
  Let $T$ be a maximal subsemigroup of a finite semigroup $S$ such that
  $S\setminus T$ is contained in a regular $\J$-class of $J$ of $S$.  Suppose
  that there exist non-empty proper subsets $A \subsetneq J/\L$ and $B
  \subsetneq J/\R$ such that $T\cap J$ is the union of the sets in $A$ and in
  $B$. Then $T$ is generated by any set consisting of:
  \begin{enumerate}[label=\emph{(\roman*)},ref=\roman*]
    \item
      $X \setminus J$;

    \item
      a generating set for the ideal of $S$ consisting of those $\J$-classes
      that lie strictly below $J$ in the $\J$-class partial order;

    \item 
      a generating set for a single group $\H$-class $H_x$ in an $\L$-class
      belonging to $A$ for some $x \in T\cap J$;

    \item 
      for every source vertex $U$ in the induced subdigraph of $\Gamma_{\L}$ on
      $A$, one element $y$ such that $y\R x$ and where $L_y$ belongs to $U$;
    
    \item 
      for every source vertex $V$ in the induced subdigraph of $\Gamma_{\R}$ on
      the complement of $B$, one element $z$ such that $z\L x$ and where $R_z$
      belongs to $V$;

    \item 
      a generating set for a single group $\H$-class $H_{x'}$ in an $\R$-class
      belonging to $B$ for some $x' \in T\cap J$;

    \item 
      for every source vertex $U'$ in the induced subdigraph of $\Gamma_{\R}$ on
      $B$, one element $y'$ such that $y'\L x'$ and where $R_{y'}$ belongs to
      $U'$;

    \item 
      for every source vertex $V'$ in the induced subdigraph of $\Gamma_{\L}$ on
      the complement of $A$, one element $z'$ such that $z'\R x'$ and where
      $L_{z'}$ belongs to $V'$;
    
    \item
      for every source vertex $U$ of $\Gamma_{\L}$ contained in 
      $A$, one element $t$ such that $R_t \in B$ and $L_t \in U$. 
  \end{enumerate}
\end{prop}
\begin{proof}
  Let $Y$ be a set of the kind described in the proposition.
  Clearly every element of $Y$ is contained in $T$, and so $\genset{Y}$ is a
  subsemigroup of $T$. Furthermore, the inclusion of the generators in (i) and
  (ii) implies that $\genset{Y}$ contains $S \setminus J$.  To show that $T \leq
  \genset{Y}$, let $a \in T \cap J$.  By the definition of $T$, either $L_{a}
  \in A$ or $R_{a} \in B$.
  
  First suppose that $L_{a} \in A$ and $R_{a} \notin B$.  The vertex of
  $\Gamma_{\L}$ containing $L_{a}$ is reachable from some vertex $U$ that is a
  source of the induced subdigraph of $\Gamma_{\L}$ on $A$. Hence there exists
  an element $y \in Y$ of type (iv) such that $L_{y} \in U$ and $y \R x$, where
  $x$ is an element in the group $\H$-class from part (iii).  By
  Lemma~\ref{lem-path-gamma}, either $L_{a} = L_{y}$, or there exists $r' \in S
  \setminus J \subseteq \genset{Y}$ such that $yr' \in L_{a}$ and,
  by Lemma~\ref{lem-d-finite}, $yr' \in R_{y} = R_{x}$. In either
  case, there exists an element $r \in \genset{Y}$ such that $r \R x$ and $r \L
  a$.  Likewise, by using a generator of type (v), there exists an element $s
  \in \genset{Y}$ such that $s \L x$ and $s \R a$.  Since $H_{x}$ is a group, it
  follows by Green's Lemma that $a \in H_{a} = sH_{x}r \subseteq \genset{Y}$.
  If instead $L_{a} \notin A$ and $R_{a} \in B$, then the proof that $a \in
  \genset{Y}$ is similar.

  For the final case, suppose that $L_{a} \in A$ and $R_{a} \in B$. The vertex
  of $\Gamma_{\L}$ containing $L_{a}$ is reachable from some vertex $U$ that is
  a source of $\Gamma_{\L}$.  If $U \subseteq A$ then by (ix) there exists an
  element $t \in Y$ such that $R_{t} \in B$ and $L_{t} \in U$. If $U
  \not\subseteq A$ then, by the previous paragraph, every element $t'$
  such that $R_{t'} \in B$ and $L_{t'} \notin A$ is contained in $\genset{Y}$.
  In either case, there exists an element $t \in \genset{Y}$ such that
  $R_{t}\in B$ and $L_t\in U$.
  Therefore, since $L_{a}$ is reachable from $L_{t}$ in $\Gamma_{\L}$, there
  exists an element $r \in \genset{Y}$ such that $r \in L_{a} \cap R_{t}$.  By
  the regularity of $J$, there exists an idempotent $e$ such that $e \R r$, and
  since $A \cup B$ corresponds to an independent set in $\Delta$, $L_{e}
  \not\in A$. By the arguments of the above paragraph, the $\H$-classes $H_{e}$
  and $H_{s} = L_{e} \cap R_{a}$ are both contained in $\genset{Y}$. Since
  $H_{e}$ is a group, we have $a \in H_{a} = sH_{e}r \subseteq \genset{Y}$.
  
  Note that if there exists an idempotent $e \in S$ such that $L_{e} \not\in A$
  and $R_{e} \not\in B$ (such an idempotent exists if and only if the
  complement of $A \cup B$ corresponds to a non-independent set in $\Delta$),
  then $a \in H_{a} = (L_{a} \cap R_{e})(R_{a} \cap L_{e})$. Hence if such an
  idempotent $e$ exists, the generators in (ix) are redundant.
\end{proof}


\subsubsection{Maximal subsemigroups from removing either $\R$-classes or
$\L$-classes: types~\eqref{item-l} and~\eqref{item-r}}\label{subsec-row-col}

If the set of $\R$-classes $B$ is empty, then the criteria for $T$ to be a
subsemigroup of $S$ can be simplified from those in Proposition~\ref{prop-rect}.
In particular, the second part of condition~\eqref{item-prop-rect-1} and
condition~\eqref{item-prop-rect-3} are vacuously satisfied.
Condition~\eqref{item-prop-rect-5} of Proposition~\ref{prop-rect} is also
automatically satisfied.  Furthermore, the reference to the set $B$ in
condition~\eqref{item-prop-rect-4} can be removed, to obtain the following
immediate corollary.

\begin{cor}\label{cor-remove-l}
  Suppose that $A \neq \varnothing$ and $B = \varnothing$. Then $T$ is a
  (proper) subsemigroup of $S$ if and only if
  conditions~\eqref{item-prop-rect-1} and~\eqref{item-prop-rect-2} of
  Proposition~\ref{prop-rect} hold, and every vertex $V$ of $\Gamma_{\L}$ with
  $c(V) = 1$ is contained in $A$.
\end{cor}

Thus we may give necessary and sufficient conditions for a
maximal subsemigroup of type~\eqref{item-l} to exist.

\begin{cor}\label{cor-l}
  Suppose that $A \neq \varnothing$ and $B = \varnothing$. Then $T$ is a maximal
  subsemigroup of $S$ if and only if the complement of $A$ is a source of
  $\Gamma_{\L}$ with colour $0$, and there is no maximal subsemigroup of $S$
  of type~\eqref{item-both} whose subset of $J / \L$ is equal to $A$.
\end{cor}
\begin{proof}
  ($\Rightarrow$)
  Since $\Gamma_{\L}$ is acyclic, it follows that any induced subdigraph of
  $\Gamma_{\L}$ is acyclic. Therefore there is a sink in the induced subdigraph
  of $\Gamma_{\L}$ on the vertices not contained in $A$. Note that this vertex
  is not necessarily a sink in $\Gamma_{\L}$ itself. Let $A^{\prime}$ be the
  subset of $J / \L$ formed from the union of $A$ with the set of $\L$-classes
  contained in this sink, and define $T'$ to be the subset of $S$ such that $S
  \setminus T' \subseteq J$ and $T' \cap J$ is the union of the $\L$-classes in
  $A'$.  Then either $A' = J / \L$, or, by Corollary~\ref{cor-remove-l}, $T'$ is
  a proper subsemigroup of $S$ that properly contains $T$.  Since $T$ is
  maximal, it follows that $A' = J / \L$, and so the complement of $A$
  forms a single vertex of $\Gamma_{\L}$.  This vertex is a source of
  $\Gamma_{\L}$ by condition~\eqref{item-prop-rect-2} of
  Proposition~\ref{prop-rect}, and has colour $0$ by
  Corollary~\ref{cor-remove-l}.
  
  Since $T$ is a maximal subsemigroup of $S$, it is not 
  contained in another maximal subsemigroup of $S$. Since $T$ is not a
  maximal subsemigroup of type~\eqref{item-both}, $T$ is not contained in such
  a maximal subsemigroup.

  ($\Leftarrow$)
  Let $M$ be a maximal subsemigroup of $S$ containing $T$.  Since $T$ contains a
  union of $\L$-classes of $J$, $M$ must be of type~\eqref{item-both} or
  type~\eqref{item-l}.  In either case, it follows that $M$ contains the
  $\L$-classes in $A$. However, $A$ lacks only one vertex of $\Gamma_{\L}$, and
  since a maximal subsemigroup is a proper subsemigroup, it follows that $M$
  contains no additional $\L$-classes. Such a maximal subsemigroup of
  type~\eqref{item-both} does not exist by assumption, and so $M$ is of
  type~\eqref{item-l}. Therefore $T = M$.
\end{proof}

Analogues of Corollaries~\ref{cor-remove-l} and~\ref{cor-l} hold 
in the case that $A = \varnothing$ and $B \not = \varnothing$, giving necessary
and sufficient conditions for a maximal subsemigroup of type~\eqref{item-r} to
exist. 

We describe how to use Corollary~\ref{cor-l} to compute the maximal
subsemigroups of type~\eqref{item-l} corresponding to a regular $\J$-class.
The first step is to compute the maximal subsemigroups of
type~\eqref{item-both}, as described above, and in doing so, to record the sets
$A$ and $B$ that occur for each maximal subsemigroup.  The second step is to
search for the sources of the digraph $\Gamma_{\L}$ with colour $0$. For each
source, we check whether its complement $A$ occurs as the set of $\L$-classes
of some maximal subsemigroup of type~\eqref{item-both}.  Given a set $A$
satisfying the conditions in Corollary~\ref{cor-l}, the final step is to
specify a generating set for the maximal subsemigroup $T$; this is described in
Proposition~\ref{prop-genset-l}. 

\begin{ex}\label{ex-l-r}
  Let $W$ and $J$ be the semigroup and the $\J$-class, respectively, from
  Example~\ref{ex-sec-3}. The digraph $\Gamma_{\L}$, depicted in
  Figure~\ref{fig-gamma}, contains two sources of colour $0$, $\{ L_{x_{3}}
  \}$ and $\{ L_{x_{4}} \}$.  The
  complements of these vertices in $J / \L$ are $A_{1} = \{ L_{x_{1}},
  L_{x_{4}}, L_{x_{1}x_{6}} \}$ and $A_{2} = \{ L_{x_{1}}, L_{x_{3}},
  L_{x_{1}x_{6}} \}$, respectively.  In Example~\ref{ex-rect}, we found no
  maximal subsemigroup of $W$ of type~\eqref{item-both} whose set of
  $\L$-classes was equal to $A_{1}$ or $A_{2}$. Therefore there are two maximal
  subsemigroups of $W$ of type~\eqref{item-l} arising from $J$: the set
  consisting of $W \setminus J$ along with the union of the $\L$-classes in
  $A_{1}$, and the set consisting of $W \setminus J$ along with the union of the
  $\L$-classes in $A_{2}$.
    
  In $\Gamma_{\R}$, the sources are $\{ R_{x_{1}}\}$ and $\{ R_{x_{2}} \}$, and
  they have colour $0$.  In Example~\ref{ex-rect}, we found no maximal
  subsemigroup of $W$ of type~\eqref{item-both} whose set of $\R$-classes was
  equal to the complement of either of these sources.  Hence $W\setminus
  R_{x_1}$ and $W\setminus R_{x_2}$ are the only maximal subsemigroups of $W$ of
  type~\eqref{item-r} arising from $J$.
\end{ex}

The following proposition describes generating sets for maximal subsemigroups
of type~\eqref{item-l}, the proof is similar to (but less complicated than) the
proof of Proposition~\ref{prop-genset-rect}, and is omitted.

\begin{prop}\label{prop-genset-l}
  Let $T$ be a subsemigroup of a finite semigroup $S$ such that $S
  \setminus T$ is contained in a regular $\J$-class $J$ of $S$. Suppose that
  there exists a non-empty proper subset $A \subsetneq J / \L$ such that $T \cap
  J$ is the union of the $\L$-classes in $A$.  Then $T$ is generated by any set
  consisting of:
  \begin{enumerate}[label=\emph{(\roman*)},ref=\roman*]
    \item
      $X \setminus J$;

    \item
      a generating set for the ideal of $S$ consisting of those $\J$-classes
      that lie strictly below $J$ in the $\J$-class partial order;

    \item
      a generating set for a single group $\H$-class $H_{x}$ in an $\L$-class
      belonging to $A$ for some $x \in T\cap J$;

    \item
      for every source vertex $U$ in the induced subdigraph of $\Gamma_{\L}$ on
      $A$, one element $y$ such that $y \R x$ and where $L_{y}$ belongs to $U$;
      
    \item
      for every source vertex $V$ in $\Gamma_{\R}$, one element $z$ such that $z
      \L x$ and where $R_{z}$ belongs to $V$.

  \end{enumerate}
\end{prop}

Generating sets for maximal subsemigroups of type~\eqref{item-r} are obtained
analogously.


\subsubsection{Maximal subsemigroups from removing the $\J$-class:
type~\eqref{item-empty}}\label{subsec-empty}

If the set of $\L$-classes $A$ and the set of $\R$-classes $B$ are both empty,
then $T = S \setminus J$, and the criteria for $T$ to be a subsemigroup of $S$
can be further simplified from those in Proposition~\ref{prop-rect}.  In
particular,
conditions~\eqref{item-prop-rect-1},~\eqref{item-prop-rect-2},~\eqref{item-prop-rect-3},
and~\eqref{item-prop-rect-5} of Proposition~\ref{prop-rect} are vacuously
satisfied, leaving only condition~\eqref{item-prop-rect-4}, which permits the
following corollary.

\begin{cor}
  The subset $S \setminus J$ of $S$ is a subsemigroup if and only if
  $\Theta$ has no edges.
\end{cor}

To compute the maximal subsemigroups of type~\eqref{item-empty} corresponding to
a regular $\J$-class, it is necessary to have first computed those maximal
subsemigroups of types~\eqref{item-every},~\eqref{item-both},~\eqref{item-l},
and~\eqref{item-r}. If no such maximal subsemigroups exist, then we check the
number of edges of the graph $\Theta$; if $\Theta$ has no edges, then $S
\setminus J$ is a maximal subsemigroup of $S$. In this case, a generating set
for $S \setminus J$ is given by the union of $X \setminus J$ with a generating
set for the ideal of $S$ consisting of those $\J$-classes that lie strictly
below $J$ in the $\J$-class partial order.

\subsection{The algorithm}

In this section we describe the overall algorithm for computing the maximal
subsemigroups of a finite semigroup $S$. This is achieved by putting together
the procedures described in the preceding sections. 

The algorithm considers, in turn, each $\J$-class that contains a generator. 
It is clear that it is only necessary to consider those $\J$-classes containing
generators. Maximal and non-maximal $\J$-classes are
then treated separately.  Every maximal $\J$-class contains at least one
generator.  A non-regular maximal $\J$-class is necessarily trivial.  For any
maximal $\J$-class, $S \setminus J$ is a subsemigroup of $S$.  Therefore for a
trivial maximal $\J$-class, the only maximal subsemigroup arising from $J$ is
$S \setminus J$.  This leaves the non-trivial maximal $\J$-classes to consider.
Such a $\J$-class is necessarily regular. Hence a maximal subsemigroup of $S$
arising from a non-trivial maximal $\J$-class $J$ is of the form $(S \setminus
J) \cup U$, and occurs precisely when $U \cup \{0\}$ is a maximal subsemigroup
of $J^{*}$. Thus it suffices to compute the maximal subsemigroups
of the Rees 0-matrix semigroup isomorphic to $J^*$ of
types~\eqref{item-remove-lambda}--\eqref{item-subgroup}, as described in
Section~\ref{section-rees}.  For the non-maximal $\J$-classes we proceed as
described earlier in this section. 

The algorithm is described in pseudocode in Algorithm~\ref{algorithm-arbitrary}.
The algorithms described in this paper are fully implemented in the {\sc
Semigroups} package~\cite{Mitchell2017aa} for {\sc GAP}~\cite{GAP4}, and the
underlying algorithms for graphs and digraphs are implemented in the {\sc
Digraphs} package~\cite{DeBeule2017aa} for {\sc GAP}~\cite{GAP4}. 

\begin{algorithm}
  \caption{Maximal subsemigroups of a finite semigroup}
  \label{algorithm-arbitrary}
  \begin{algorithmic}[1]

    \item[\textbf{Input:}] $S = \genset{X}$, a finite semigroup with generating
      set $X$.

    \item[\textbf{Output:}] the non-empty maximal subsemigroups $\mathfrak{M}$
      of $S$.

    \item $\mathfrak{M} := \varnothing$ 

    \For{$J_{x} \in \bigset{J \in S/\J}{J \cap X \neq \varnothing}$}

      \If{$J_{x}$ is a maximal $\J$-class of $S$}

        \If{$|J_{x}| = 1$}
        \Comment{$J_{x}$ is non-regular, or is a trivial subgroup}

          \State{$\mathfrak{M} \gets \mathfrak{M} \cup \{S \setminus J_{x}\}$}

        \Else{}
        \Comment{$J_{x}$ is necessarily regular}

          \State{compute $\phi$, an isomorphism from $J_{x}^{*}$ to a normalized
          Rees $0$-matrix semigroup}

          \State{compute $\mathfrak{J}$, the maximal subsemigroups of
          $(J_{x})\phi \cup \{0\}$ of
          types~\eqref{item-remove-lambda},~\eqref{item-remove-i},~\eqref{item-rect},
          and~\eqref{item-subgroup}}\Comment{Section~\ref{section-rees}}

          \For{$M \in \mathfrak{J}$}

          \State{$\mathfrak{M} \gets \mathfrak{M} \cup \left\{ (S \setminus
          J_{x}) \cup (M \setminus \{0\})\phi^{-1} \right\}$}

          \EndFor{}

        \EndIf{}

      \Else{}
      \Comment{$J_{x}$ is a non-maximal $\J$-class}

        \State{$X^{\prime} := \bigset{y \in X}{J_{y} > J_{x}}$}
      
        \If{$J_{x}$ is non-regular and $x \notin \genset{X^{\prime}}$}

          \State{$\mathfrak{M} \gets \mathfrak{M} \cup \{ S \setminus J_{x} \}$}
          \Comment{Type~\eqref{item-nonreg},
          Proposition~\ref{prop-non-regular-arbitrary}}

          \ElsIf{$J_{x}$ is regular and $J_x\cap X \not\subseteq \genset{X'}$}

            \State{compute the digraphs $\Gamma_{\L}$ and $\Gamma_{\R}$ and the
            graphs $\Delta$ and $\Theta$}

            \State{compute $\phi$, an isomorphism from $J_{x}^{*}$ to a
            normalized Rees 0-matrix semigroup}

            \State{compute the set $E$ consisting of one idempotent in each
            $\L$-class of $J_x$}

            \State{compute $\mathfrak{J}$, the maximal subsemigroups of
            $(J_{x})\phi \cup \{0\}$ of type~\eqref{item-subgroup} which contain
            $(EX^{\prime})\phi$}
            \Comment{Algorithm~\ref{algorithm-type-6}}

            \For{$M \in \mathfrak{J}$}

              \State{$\mathfrak{M} \gets \mathfrak{M} \cup \left\{ (S \setminus
              J_{x}) \cup (M \setminus \{0\})\phi^{-1} \right\}$}
              \Comment{Type~\eqref{item-every}, Corollary~\ref{cor-intersect}}

            \EndFor{}

            \State{add maximal subsemigroups of type~\eqref{item-both} to
            $\mathfrak{M}$}
            \Comment{Section~\ref{subsec-rect}}

            \State{add maximal subsemigroups of types~\eqref{item-l}
            and~\eqref{item-r} to $\mathfrak{M}$}
            \Comment{Section~\ref{subsec-row-col}}

            \If{no maximal subsemigroups have been found to arise from $J_{x}$}

              \State{$\mathfrak{M} \gets \mathfrak{M} \cup \{S \setminus
              J_{x}\}$} \Comment{Type~\eqref{item-empty},
              Section~\ref{subsec-empty}}

          \EndIf{}

        \EndIf{}

      \EndIf{}

    \EndFor{}

    \State{\Return{$\mathfrak{M}$}}
  \end{algorithmic}
\end{algorithm}

\section{Performance analysis}
\label{section-performance}

In this section we provide some analysis of the performance of the algorithms
described in this paper.  These algorithms are implemented in the
\textsc{Semigroups} package~\cite{Mitchell2017aa} for GAP~\cite{GAP4} and the
computations in this section were run on a 2.66~GHz Intel Core i7 processor with
8GB of RAM. 

Given a semigroup $S$ represented by a set of generators $X$, such as
transformations, we compare the time taken to compute
the Green's structure of $S$ with that taken to find the maximal
subsemigroups. Additionally we include the amount of time spent during the
computation of the maximal subsemigroups on finding maximal 
cliques or maximal subgroups of group $\mathscr{H}$-classes, when these times
are not negligible. 

The semigroups considered are given below. 
\begin{itemize}
  \item 
    The full transformation monoids $\mathcal{T}_n$ consisting of all
    transformations of $\{1, \ldots, n\}$ when $n = 2, \ldots, 11$. If $n
    > 2$, then $\mathcal{T}_n$ is generated by 3 transformations, and this is
    the minimal number.  There are $n ^ n$ transformations of $\{1, \ldots,
    n\}$ and the number of $\mathscr{J}$-classes in $\mathcal{T}_n$ is $n$.
    See Example~\ref{ex-sec-3} for further details.

  \item 
    The inverse monoids $\mathcal{PORI}_n$ consisting of the bijective functions 
    between subsets of $\{1, \ldots, n\}$ which preserve or reverse the
    cyclic orientation of $\{1, \ldots, n\}$ for $n = 11, \ldots, 20$.
    As an inverse monoid $\mathcal{PORI}_n$ is minimally generated 
    by 3 elements when $n > 2$. If $n > 0$, then 
    $$|\mathcal{PORI}_n| = 1 + n \binom{2n}{n} - \frac{n ^ 2 (n ^ 2 - 2n +
    3)}{2}$$
    and the number of $\mathscr{J}$-classes in $\mathcal{PORI}_n$ is $n + 1$. 

  \item 
    The \textit{Jones monoids} $\mathcal{J}_n$, introduced in~\cite{Lau2006aa},
    and also known as the \emph{Temperley-Lieb monoid}, for $n = 6, \ldots,
    20$.  The definition of $\mathcal{J}_n$ is too long to give here.  A
    minimal generating set for $\mathcal{J}_n$ has $n - 1$ elements when $n >
    1$.  If $n > 0$, then 
    $$|\mathcal{J}_n| = \frac{1}{n + 1}\binom{2n}{n}$$
    and the number of $\mathscr{J}$-classes in $\mathcal{J}_n$ is 
    $\left\lceil \frac{n}{2}\right\rceil$.

  \item 
    100 subsemigroups of $\mathcal{T}_9$ generated by $9$ transformations chosen
    uniformly at random; see Table~\ref{table-random} for further information.

  \item 
    100 regular Rees 0-matrix semigroups $R = \M^ 0 [I, G, \Lambda; P]$ where
    $|I| = |\Lambda| = 1, \ldots, 20$, $G$ was a permutation group on $10$
    points chosen uniformly at random from the representatives of conjugacy
    classes of all subgroups of the symmetric group on $10$ points, and the
    entries of the matrix $P$ were chosen randomly such that the
    Graham-Houghton graph of $R$ had between $1$ and $|I| = |\Lambda|$
    connected components. The entries of $P$ were not chosen uniformly, nor do
    we claim that the resulting Rees 0-matrix semigroups represent a uniform
    sample of such semigroups. 

\end{itemize}
We refer the interested reader to~\cite{East2017aa} and the references therein
for further details about the monoids $\mathcal{T}_n$, $\mathcal{PORI}_n$, and
$\mathcal{J}_n$ and for a characterisation of their maximal subsemigroups;  see
also Tables~\ref{table-T},~\ref{table-PORI}, and~\ref{table-J} for the numbers
of maximal subsemigroups of these monoids.  The results obtained
in~\cite{East2017aa} relied heavily on computational experiments performed
using the algorithms described in this paper, and their implementations in
GAP~\cite{GAP4}.
The monoids $\mathcal{T}_n$, $\mathcal{PORI}_n$, and $\mathcal{J}_n$ were
chosen because they are well-studied in the literature, because they have
different representations in the \textsc{Semigroups} package for GAP, and
because they exhibit different behaviour when computing the maximal
subsemigroups. In Figures~\ref{fig-T},~\ref{fig-PORI}, and~\ref{fig-J}, the
time taken to compute the partial order of the $\mathscr{J}$-classes of a given
semigroup is compared to that taken to find the maximal subsemigroups.  The
partial order of the $\mathscr{J}$-classes of a semigroup was obtained using
the method described in~\cite[Algorithm 14]{East2015ab} and implemented in the
\textsc{Semigroups} package~\cite{Mitchell2017aa} for GAP~\cite{GAP4}.  The
range of values of $n$, in each case, includes the largest value where the
Green's structure could be computed within the limitations of the
hardware.\footnote{Note that $|\mathcal{T}_{11}| = 285\ 311\ 670\ 611$,
$|\mathcal{PORI}_{20}| = 2\ 756\ 930\ 503\ 801$, and
$|\mathcal{J}_{20}| = 6\ 564\ 120\ 420$.} 

We opted to include some data relating to random semigroups to highlight
possible typical behaviour of the maximal subsemigroups algorithms. What
constitutes a reasonable notion of a ``random semigroup'' is debatable,
although we believe that the notions used here are somewhat meaningful.
Figure~\ref{fig-random-1} concerns the 100 subsemigroups of $\mathcal{T}_9$
generated by $9$ transformations chosen uniformly at random.  A point on the
$x$-axis corresponds to a single such semigroup. The points on the $x$-axes are
sorted in increasing order according to the ratio of the time taken to compute
the maximal subsemigroups and the time taken to compute the partial order of the
$\mathscr{J}$-classes. The particular choices of $9$ transformations on $9$
points were made because they approach the limit of what is practical to
compute.  Figure~\ref{fig-random-2} concerns the 2000 random Rees 0-matrix
semigroups.
The Green's structure of a regular Rees 0-matrix semigroup
can be determined immediately from its definition, and so the time to determine
this is not used for comparison in Figure~\ref{fig-random-2}. For each
dimension considered, the mean of 100 examples is shown in
Figure~\ref{fig-random-2}. 

While there are some instances in the data presented in this section where
computing the maximal subsemigroups is several orders of magnitude slower than
computing the Green's structure, for the majority the times taken are roughly
comparable. For $\mathcal{T}_n$ and $\mathcal{PORI}_n$, the time taken to
compute the maximal subsemigroups is not dominated by either the time taken to
compute maximal cliques or the time to compute maximal subgroups, but rather by
constructing and processing the graphs $\Gamma_{\L}$, $\Gamma_{\R}$, $\Delta$,
and $\Theta$ from Section~\ref{section-arbitrary}.  For these monoids, the
graphs $\Delta$ have 2 vertices, and as such there are no maximal cliques
from which a maximal subsemigroup could arise.  On the other hand, for the
Jones monoids $\mathcal{J}_n$, the majority of the time in the computation of
the maximal subsemigroups is spent finding maximal cliques. 

For Rees 0-matrix semigroups, it appears from Figure~\ref{fig-random-2} that
the time taken to compute maximal subsemigroups approaches the time taken to
compute maximal cliques as the dimension increases. 

\begin{table}
  \begin{center}
    \begin{tabular}{l|r|r|r|r|r|r|r|r|r|r}
      Degree $n$ & 2  & 3  & 4  & 5 & 6  & 7  & 8  & 9  & 10 & 11 \\\hline
      Maximal subsemigroups 
      & 2 & 5 & 9 & 23 & 54 & 185 & 354 & 1377 & 3978 & 363905
    \end{tabular}
    \caption{The number of maximal subsemigroups of the full
    transformation monoid $\mathcal{T}_{n}$
     for some small values of $n\in \N$.}
    \label{table-T}
  \end{center}
\end{table}

\begin{table}
  \begin{center}
    \begin{tabular}{l|r|r|r|r|r|r|r|r|r|r|r}
      Degree $n$ & 10 & 11 & 12 & 13 & 14 & 15 & 16 & 17 & 18 & 19 & 20
      \\\hline
      Maximal subsemigroups & 9 & 14 & 7 & 16 & 11 & 11 & 5 & 19 & 7
      & 22 & 9 
    \end{tabular}
    \caption{The number of maximal subsemigroups of $\mathcal{PORI}_{n}$
     for some small values of $n\in \N$.}
    \label{table-PORI}
  \end{center}
\end{table}

\begin{table}
  \begin{center}
    \begin{tabular}{l|r|r|r|r|r|r|r|r|r|r|r|r|r|r|r}
      Degree $n$ & 10 & 11 & 12 & 13 & 14 & 15 & 16 & 17 & 18 &
      19 & 20 \\\hline
      Maximal subsemigroups & 85 & 129 & 199 &
      311 & 491 & 781 & 1249 & 2005 & 3227 & 5203 & 8399 \\
    \end{tabular}
    \caption{The number of maximal subsemigroups of the Jones
    monoids $\mathcal{J}_{n}$ for some small values of $n\in \N$.}
    \label{table-J}
  \end{center}
\end{table}

\begin{table}
  \begin{center}
    \begin{tabular}{l|r|r|r|r|r|r}
      & Min & Max & Mean & Median & Standard deviation\\\hline
      Size & 125333 & 85014449 & 5657333 & 2484319 & 9865603 \\
      Number of maximal subsemigroups & 9 & 640 & 25 & 13 & 65 \\
      Number of $\mathscr{J}$-classes & 8 & 89858 & 5336 & 2550  & 10187\\
      Time for maximal subsemigroups (milliseconds) & 8 & 272132 & 15998 & 5003
      & 37944
    \end{tabular}
    \caption{Information about the 100 subsemigroups of $\mathcal{T}_9$
    generated by $9$ transformations chosen uniformly at random.}
    \label{table-random}
  \end{center}
\end{table}

\begin{figure}
  \centering
  \includegraphics[width=0.6\textwidth]{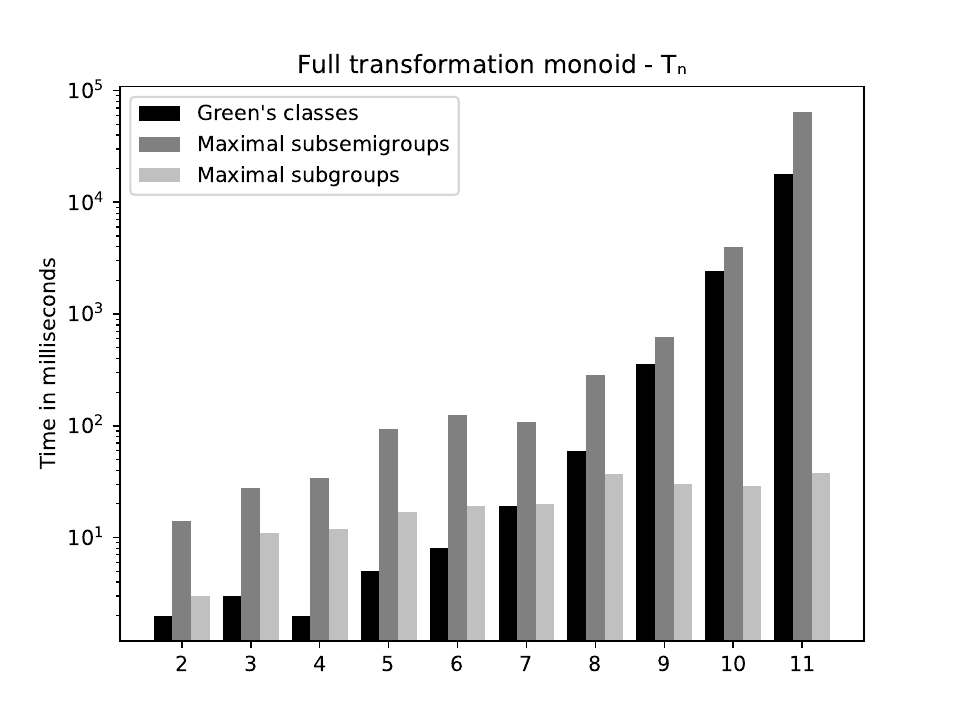}
  \caption{Comparison of the time taken to compute the partial order of
  $\mathscr{J}$-classes of the full transformation monoids $\mathcal{T}_n$, $n
  = 2, \ldots, 11$, with the time taken to compute their maximal subsemigroups,
  and the time spent finding maximal subgroups of group $\mathscr{H}$-classes.} 
  \label{fig-T}
\end{figure}
\begin{figure}
  \centering
  \includegraphics[width=0.6\textwidth]{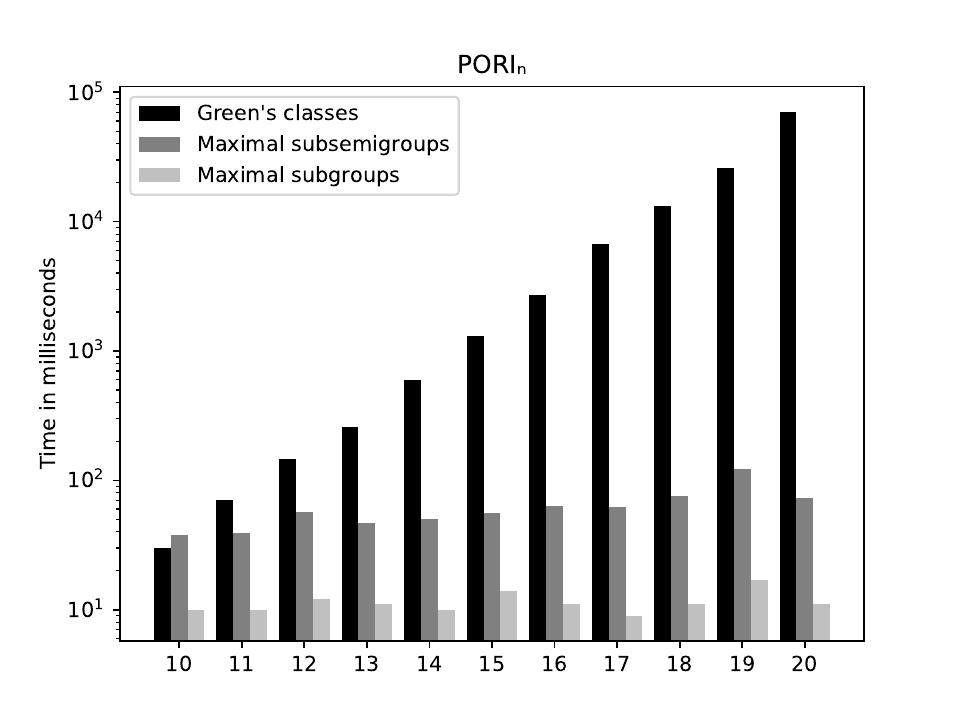}
  \caption{Comparison of the time taken to compute the partial order of
  $\mathscr{J}$-classes of the monoids $\mathcal{PORI}_n$, $n = 10,
  \ldots, 20$, with the time taken to compute their maximal subsemigroups, and
  the time spent finding maximal subgroups of group $\mathscr{H}$-classes.} 
  \label{fig-PORI}
\end{figure}
\begin{figure}
  \centering
  \includegraphics[width=0.6\textwidth]{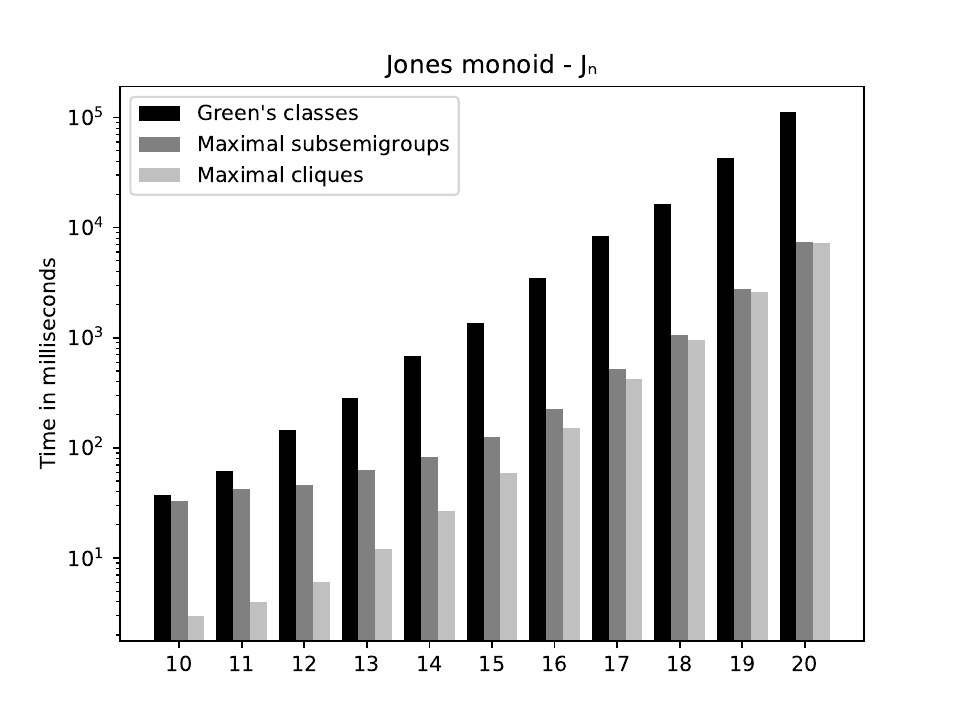}
  \caption{Comparison of the time taken to compute the partial order of
  $\mathscr{J}$-classes of the Jones monoids $\mathcal{J}_n$, $n = 10,
  \ldots, 20$, with the time taken to compute their maximal subsemigroups, and
  the time spent finding maximal cliques.} 
  \label{fig-J}
\end{figure}
\begin{figure}
  \centering
  \includegraphics[width=0.6\textwidth]{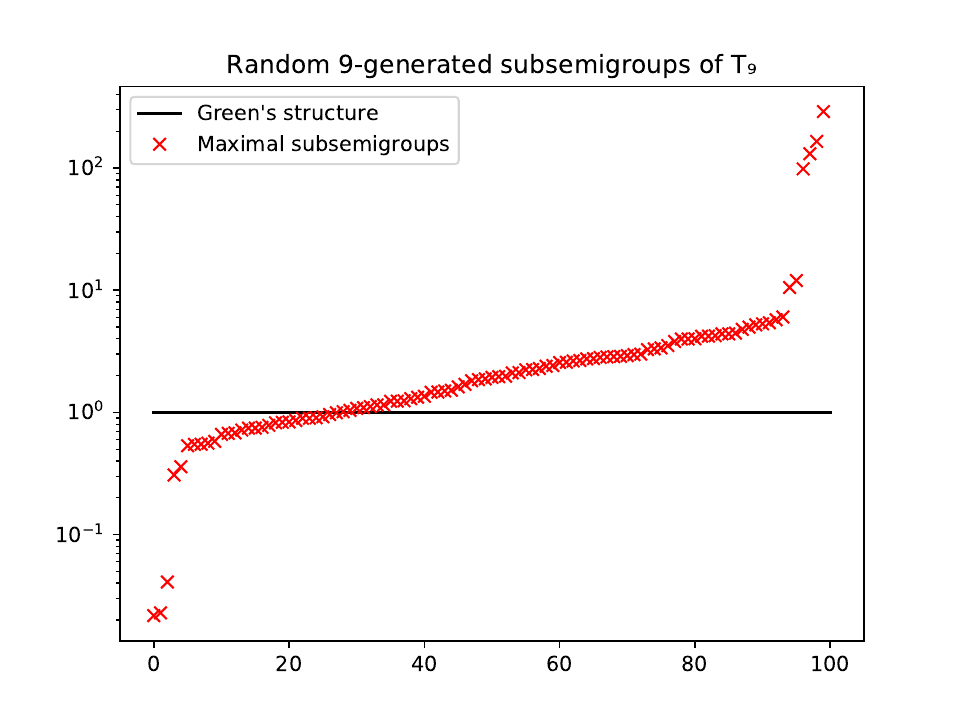}
  \caption{Comparison of the time taken to compute the partial order of
  $\mathscr{J}$-classes of 100 random 9-generated subsemigroups of
  $\mathcal{T}_9$ 
  with the time taken to compute their maximal subsemigroups.}
  \label{fig-random-1}
\end{figure}
\begin{figure}
  \centering
  \includegraphics[width=0.6\textwidth]{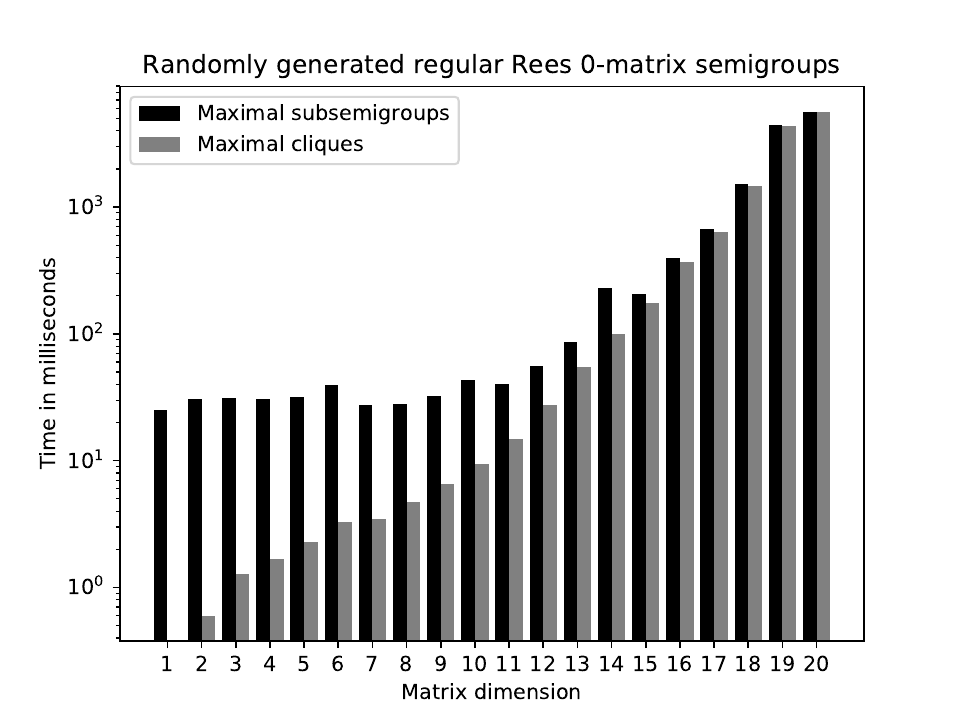}
  \caption{Comparison of the mean time taken to compute the maximal
  subsemigroups of 100 random regular Rees 0-matrix semigroups of a given
  dimension with the mean time taken to compute the maximal cliques of the
  duals of their Graham-Houghton graphs.}
  \label{fig-random-2}
\end{figure}

\section*{Acknowledgements}
The first author wishes to thank the support of the School of Mathematics and
Statistics at the University of St Andrews for his Ph.D. Scholarship.
The third author wishes to acknowledge the support of his Carnegie Ph.D.
Scholarship from the Carnegie Trust for the Universities of Scotland.
The authors also thank the anonymous referee for their helpful comments.

\end{document}